\documentclass[10pt,a4paper]{article}

\usepackage{amsmath,amssymb,amsthm}
\usepackage{xcolor}
\usepackage{url}
\usepackage{authblk}
\usepackage{tikz-cd}
\usepackage[normalem]{ulem}
\usepackage{algorithm}
\usepackage{algpseudocode}

\usepackage[colorlinks=true, linkcolor=blue, citecolor=blue, urlcolor=blue]{hyperref}

\usepackage{etoolbox}
\patchcmd{\thebibliography}{\section}{\subsection}{}{}
\makeatletter
\g@addto@macro\@openbib@code{\setlength{\itemsep}{0pt}}
\makeatother

\usepackage[T1]{fontenc}
\newtheorem{theorem}{Theorem}[section]
\newtheorem{lemma}{Lemma}[section]
\newtheorem{rmk}[theorem]{Remark}

\newtheorem{cor}{Corollary}

\definecolor{darkgreen}{rgb}{0,0.5,0}

\author[1, 4]{Luciano N. Grippo \thanks{lgrippo@campus.ungs.edu.ar}}
\author[2, 4]{Min Chih Lin \thanks{oscarlin@dc.uba.ar}}
\author[3, 4]{Camilo Vera \thanks{cvera@ic.fcen.uba.ar}}
\affil[1]{Universidad Nacional de General Sarmiento. Instituto de Ciencias; Argentina.}
\affil[2]{Instituto de Cálculo and Departamento de Computación, Universidad de Buenos Aires, Argentina.}
\affil[3]{Instituto de Cálculo and Departamento de Matemática, Universidad de Buenos Aires, Argentina.}
\affil[4]{Consejo Nacional de Investigaciones Científicas y Técnicas, Argentina.}

\title{Perfect Edge Domination in $P_6$-free Graphs and in Graphs Without Efficient Edge Dominating Sets}
\date{}

\begin{document}

\maketitle

\begin{abstract}
	An edge of a graph dominates itself along with any edge that shares an endpoint with it. An efficient edge dominating set (also called a dominating induced matching, DIM) is a subset of edges such that each edge of the graph is dominated by exactly one edge in the subset. A perfect edge dominating set is a subset of edges in which every edge outside the subset is dominated by exactly one edge within it. In this article, we establish the NP-completeness of deciding whether a graph that does not admit any efficient edge dominating set has at least two perfect edge dominating sets. We also present a cubic time algorithm designed to identify a perfect dominating set of minimal cardinality for $P_6$-free graphs. Moreover, we show how this algorithm can be adapted to handle the weighted version of the problem and to count all perfect edge dominating sets as well as DIMs in a given graph, while preserving the same time complexity.
		
	\bigskip
	
\end{abstract}

\section{Introduction}

All graphs considered in this article are undirected, finite, and have neither loops nor multiple edges. For all definitions and notations used through this section, we refer the reader to Section~\ref{sec: preliminaries}. 

Let $G$ be a graph. We say that an edge $e$ of $G$ \emph{dominates} itself and every edge sharing an endpoint with it. An \emph{edge dominating set} is a subset of edges such that each edge of $G$ is dominated by at least one edge of the subset. An \emph{induced matching} in $G$ is a subset of edges such that each edge of $G$ is dominated by at most one edge of the subset. 

An \emph{efficient edge dominating set} is a dominating induced matching (DIM). Notice that an efficient edge dominating set may not exist. A graph having no DIM is called a \emph{DIM-less graph}. The problem of determining whether a graph has an efficient edge dominating set (DIM), denoted EED, is NP-complete in general~\cite{Gr-Sl-Sh-Ho-1993}. In addition, EED remains NP-complete in many classes as bipartite graphs of degree at most three~\cite{Lozin2002} or line graphs~\cite{Ko-Ro-2003}. In~\cite{CARDOSO2011521}, the authors prove that the problem is NP-complete for the class $\mathcal{S}_k$ formed by all bipartite $\{C_3, \ldots, C_k, H_1, \ldots, H_k\}$-free graphs of vertex degree at most three. Moreover, they showed that $\mathcal{S}=\bigcap_{k\ge 0} \mathcal{S}_k$ is a limit class. Cardoso et al., by demonstrating that EED can be expressed in terms of a monadic second-order logic (see~\cite{Co-Ma-Ro-2000}), also observed that for any class of graphs with bounded clique-width, EED can be solved in polynomial time (cf.~\cite[Lemma 5]{CARDOSO2011521}). 

Since cographs (graphs with no $P_4$ as induced subgraphs) have clique-width at most two, EED is polynomial-time solvable within cographs. These results led to investigating the complexity of EED for graphs without paths with at least five vertices, which was an open problem at that time. Some progress has been made since then. In~\cite{Br-Mo-2014}, Brandst\"adt and Mosca mention that the problem can be solved in quadratic time for $P_5$-free graph, because the class of $(\textrm{gem},P_5)$-free graphs has bounded clique-width. In addition, they found a linear time algorithm for solving the minimum weighted EED in $P_7$-free graphs. A few years later, the same authors extended this result to $P_8$-free graphs by developing a polynomial time algorithm using the same approach~\cite{Br-Mo-2017}. In the last five years, they have provided a polynomial time algorithm to solve EED in $P_9$-free graphs~\cite{Brandstadtdmgt2020} and $P_{10}$-free graphs~\cite{Brandstadt2024}. Many other articles identify graph classes where the minimum weighted EED can be solved in polynomial time, such as weakly chordal graphs~\cite{Br-2010}, hole-free graphs~\cite{Br-2010, Nevries-2014-thesis}, $S_{1,1,1}$-free graphs~\cite{CARDOSO2011521}, $S_{2,2,2}$-free graphs~\cite{He-Lo-Ri_Za-deW-2018}, and $S_{1,2,3}$-free graphs~\cite{Ko-Lo-Pu-2014}. Moreover, in \cite{Lu-Ko-Tang-2002}, a linear time algorithm is presented for solving this problem on generalized series-parallel graphs and chordal graphs. Here, $S_{i,j,k}$ denotes the graph with vertex set $\{u, x_1, \ldots, x_i, y_1, \ldots, y_j, z_1, \ldots, z_k\}$ such that $\{u, x_1, \ldots, x_i\}$ induces $P_{i+1}$, $\{u, y_1, \ldots, y_j\}$ induces $P_{j+1}$, and $\{u, z_1, \ldots, z_k\}$ induces $P_{k+1}$, with no additional edges.

Recently, Brandst\"adt and Mosca extensively explored the complexity of EED for $S_{i,j,k}$-free graphs; they proved that EED is solvable in polynomial time in all of the following graph classes: $S_{1,2,4}$-free graphs~\cite{Br-Mo-2020}, $S_{1,1,5}$-free graphs~\cite{Br-Mo-2020-3}, $S_{2,2,3}$-free graphs~\cite{Br-Mo-2020-2}. For a brief survey of known results on EED, we refer the reader to~\cite{Br2018}. For results related to the weighted version of EED see~\cite{LinMS14IPL, LinMS14}.

A \emph{perfect edge dominating set} (PED-set) is a set of edges such that each edge outside the set is dominated by exactly one edge within the set. This edge domination variant has been less studied than the efficient edge domination. Unlike a DIM, since the whole set of edges of a graph forms a PED-set of it, every graph always admits one. To the best of our knowledge, the  problem of deciding whether a graph admits a proper PED-set remains open. We refer to this decision problem as \emph{proper-PED}.

In connection with DIMs, it is not difficult to prove that a DIM in a graph is a PED-set of minimum cardinality~\cite{Lu-Ko-Tang-2002}. Lin et al. strengthened the former result by showing that the problem remains NP-complete when the input graph is a claw-free ($S_{1,1,1}$-free) graph with maximum degree at most three~\cite{Li-Lo-Mo-Sz}. In contrast, there exists a linear time algorithm to find a PED-set of minimum size for cubic claw-free graphs. Moreover, they proved that the proper-PED is polynomial-time solvable for $H$-free graphs with maximum degree at most $r$ (for some fixed $r\ge 3$) when $H$ is a linear forest, and NP-complete otherwise. This dichotomy result motivates the study of the complexity of PED for $P_k$-free graphs. When $k=5$, there exists a robust linear-time algorithm, given any input graph, either finds a minimum PED-set or exhibits an induced $P_5$ of the graph~\cite{Li-Lo-Mo-Sz, Moyano-2017-thesis}. Explicit mathematical formulations for PED, leading to exact algorithms, can be found in~\cite{doFor-Vi-Lin-Lu-Ma-Mo-Sz-2020}.

The goal of this paper is twofold: to determine the complexity of the proper-PED problem for connected DIM-less graphs, and to develop an efficient algorithm for $P_6$-free graphs, aiming to enhance the understanding of PED for $P_k$-free graphs. The paper is organized as follows. Section~\ref{sec: preliminaries} presents definitions and preliminary results used throughout the paper. Section~\ref{sec: DIM-less} contains a proof of the NP-completeness of proper-PED for connected DIM-less graphs. Section~\ref{sec: P_6-free} is devoted to presenting some results related to the correctness of the cubic-time algorithm for $P_6$-graphs. In Section~\ref{sec: the algorithm} provides a detailed description of the algorithm. Section~\ref{sec: weighted and counting version} explains how to adapt the algorithm of the previous section to solve the weighted version, and the problem of counting all PED-sets as well as DIMs. Finally, in Section~\ref{sec: conclusions}, we close the article mentioned the problem of counting DIMs in a graph.

\section{Preliminaries}~\label{sec: preliminaries}
\subsection{Efficient and perfect edge domination}
Let $G$ be a graph. The vertex and edge sets of $G$ are denoted by $V(G)$ and $E(G)$, respectively. If $uv\in E(G)$, we say that $u$ (resp. $v$) is \emph{adjacent} to $v$ (resp. $u$); we also say that  the edge $uv$ is \emph{incident} at $u$ (resp. $v$). The \emph{set of neighbors of $v$} is denoted by $N_G(v)$; i.e., all vertices $u\in V(G)$ adjacent to $v$. When the context is clear, we may write $N(v)$ instead of $N_G(v)$. We also define the \emph{closed neighborhood} of $v$ as $N_G[v]=N_G(v)\cup\{v\}$. The \emph{degree} of a vertex $v\in V(G)$ is $d_G(v)=|N_G(v)|$. A vertex $v$ is called a \emph{pendant vertex} in $G$ if $d_G(v)=1$. Two edges $e$ and $e'$ are \emph{adjacent} if they are incident at the same vertex. If $S\subseteq V(G)$, the \emph{subgraph induced} by $S$ is denote by $G[S]$, and $G-S$ stands for $G[V(G)\setminus S]$. Moreover, for a vertex $v\in V(G)$, we define $N_G^S(v)=N_G(v)\cap S$. If $H$ is a subgraph of $G$, we let $N_H(v) = N^{V(H)}_G(v)$. A subgraph $H$ of $G$ is called a \emph{spanning subgraph} of $G$ if $V(H)=V(G)$, and it is called a \emph{spanning tree} if $H$ is also a tree (i.e., a connected and acyclic graph).

We use $C_k$ and $P_k$ to denote the \emph{cycle} and the \emph{path on $k$ vertices}, respectively. The cycle $C_3$ is called a \emph{triangle}. An \emph{independent set} is a set of pairwise non-adjacent vertices. A \emph{dissociation set} is a subset of vertices in a graph where the subgraph induced  by this subset has maximum degree of at most 1. A \emph{bipartite graph} is a graph $G$ such that $V(G)$ can be partitioned into two independent sets $\{R,S\}$, called a \emph{bipartition}. A \emph{complete bipartite graph} is a bipartite graph with a bipartition $\{R,S\}$ such that every vertex in $R$ is adjacent to every vertex in $S$. We use $K_{p,q}$ to denote the complete bipartite graph with a bipartition $\{R,S\}$ such that $|R|=p$ and $|S|=q$. 

An edge $uv$ \emph{dominates} itself and any other edge adjacent to it. A subset $P\subseteq E(G)$ is a \emph{perfect edge dominating set} (PED-set) if every edge of $E(G)\setminus P$ is dominated by exactly one edge of $P$.

A perfect edge dominating set $P$ of a graph $G$ induces a partition of $V(G)$ into three subsets as follows: 

\begin{itemize}
	\item \emph{Black vertices} are those with at least two incident edges of $P$. We denote this subset of vertices by $B$.
	\item \emph{Yellow vertices} are those with exactly one incident edge of $P$. We denote this subset of vertices by $Y$.
	\item \emph{White vertices} are those with no incident edge of $P$. We denote this subset of vertices by $W$. 
\end{itemize}

An assignment of one of the three possible colors to each vertex of $G$ will be called a \emph{coloring} of $G$, and we refer to $(B,Y,W)$ as the \emph{$3$-coloring associated with $P$}. It is easy to see that the following conditions hold:
\begin{enumerate}
	\item $W$ is an independent set.
	\item $v\in Y$ if and only if $v$ is a pendant vertex in $G-W$. 
	\item If $v\in W$, then $N_G(v)\subseteq Y$.
	\item If $v\in B$, then $N_G^W(v)=\emptyset$ and $d_G(v)\ge 2$.
\end{enumerate}
Besides, a $3$-coloring $(B,Y,W)$ of a graph $G$ satisfying the four above conditions is associated with a perfect edge dominating set of $G$, consisting precisely of those edges having non-white endpoints (i.e., vertices with a non-white color assigned). A coloring is called \emph{partial} if only some vertices of $G$ have assigned colors; otherwise, it is called \emph{total}. A partial coloring is \emph{valid} if no two white vertices are adjacent, no yellow vertex has more than one non-white neighbor, and no black vertex has white neighbors. A total coloring is \emph{valid} if no two white vertices are adjacent, every yellow vertex has exactly one non-white neighbor, and no black vertex $v$ has white neighbors, and $d_G(v)\ge 2$.

Notice that, if $H$ is a subgraph of a graph $G$, and $\pi$ is a valid $3$-coloring  of $G$; i.e., a function assigning to each vertex in $V(G)$ a color black (b), yellow (y), and white (w), then $\pi$ restricted to $V(H)$ is a valid partial $3$-coloring of $H$. It is said that a valid partial $3$-coloring $\pi$ of a subgraph $H$ of $G$ \emph{can be extended to a valid $3$-coloring of $G$}, if $G$ has a valid $3$-coloring $\pi'$ such that, restricted to $V(H)$, it agrees with $\pi$, meaning that $\pi'(u)=\pi(u)$, for all $u\in V(H)\subseteq V(G)$. If $\pi$ generates a PED-set of minimum cardinality, we say that $\pi$ is an \emph{optimum valid $3$-coloring}.  

We say that a graph $G$ is an \emph{edge-weighted graph} if there exists a function $\omega:E(G)\longrightarrow\mathbb{R}$, and we refer to $\omega(e)$ as the \emph{weight} of the edge $e$. For any subset $D\subseteq E(G)$, we define its \emph{weight} as $\omega(D):=\displaystyle\sum_{e\in D}\omega(e)$. The \emph{weighted perfect edge domination problem} (WPED problem) consists of finding a PED-set $P$ of $G$ such that $\omega(P)$ is minimum. 

We say that a graph $G$ is a \emph{vertex-weighted graph} if there exists a function $\psi:V(G)\longrightarrow\mathbb{R}$, and we refer to $\psi(v)$ as the \emph{weight} of the vertex $v$. For any subset $S\subseteq V(G)$, we define its \emph{weight} as $\psi(S):=\displaystyle\sum_{v\in S}\psi(v)$.

Let $(B, Y, W)$ be the valid $3$-coloring of the weighted graph $G$ induced by an edge dominating set $P$. Define the function $\psi$ as $\psi(u) = \displaystyle \sum_{v \in N_G(u)} \omega(uv)$ for every $u \in V(G)$. Then, the weight of the PED-set $P$ in $G$ satisfies
\[
\omega(P) = \frac{\psi(V(G))}{2} - \psi(W). 
\]
It is not difficult to see that $P$ is a minimum perfect edge dominating set of $G$ if and only if $\psi(W)$ is maximum.

The \emph{PED-counting problem} for a graph $G$ consists of determining the number of PED-sets that $G$ admits. It follows  from~\cite{forte2022new} that the decision problem associated with the PED-counting problem is NP-complete (see also Lemma~\ref{lem: NFS graphs}).

There are three possible types of PED-sets in a graph $G$.
Let $P\subseteq E(G)$ be a PED-set of $G$.
\begin{enumerate}
	\item We say that $P$ is the \emph{trivial PED-set} if $P=E(G)$.
	\item We say that $P$ is a \emph{DIM} if no two edges in $P$ are adjacent.
	\item We say that $P$ is a \emph{proper PED-set} if it is neither the trivial PED-set nor a DIM.
\end{enumerate}
Now, consider a connected graph $G$.
\begin{enumerate}
	\item If $|V(G)|=1$, then $W=V(G)$ and $Y=B=\emptyset$. In this case, we have the trivial PED-set and a DIM.
	\item If $|V(G)|=2$, then $Y=V(G)$ and $W=B=\emptyset$. In this case, we have the trivial PED-set and a DIM.
	\item If $|V(G)|\geq 3$, then:
	\begin{itemize}
		\item The 3-coloring $(B,Y,W)$ induces the trivial PED-set if and only if $W=\emptyset$. Since $G$ is connected, there exists at least one vertex of degree greater than or equal two, and this vertex is black.
		\item The 3-coloring $(B,Y,W)$ induces a DIM if and only if $B=\emptyset, Y\neq\emptyset$ and $W\neq\emptyset$.
		\item The 3-coloring $(B,Y,W)$ induces a proper PED-set if and only if $B\neq\emptyset, Y\neq\emptyset$ and $W\neq\emptyset$.
	\end{itemize}
\end{enumerate}
\subsection{$P_6$-free graphs}
The following theorem is useful for our goal of developing a polynomial-time algorithm to find a PED-set of minimum cardinality in $P_6$-free graphs. Let $G$ be a graph. A \emph{dominating set} of $G$ is a set $S$ of vertices such that each vertex in $V(G) \setminus S$ is adjacent to at least one vertex in $S$. A \emph{dominating (induced) subgraph} of $G$ is a subgraph (respectively, an induced subgraph) $H$ of $G$ such that $V(H)$ is a dominating set.
\begin{theorem}\label{p6-free struct} \cite{DBLP:conf/cocoon/HofP08,DBLP:journals/dam/HofP10}
	A graph $G$ is $P_6$-free if and only if each connected induced subgraph of $G$ contains a dominating induced $C_6$ or a dominating (not necessarily induced) complete bipartite graph. Moreover, such a dominating subgraph of $G$ can be found in $O(|V|^3)$ time.
\end{theorem}
Our cubic-time algorithm in Section \ref{sec: P_6-free} for finding a PED-set of minimum cardinality for $P_6$-free graphs relies on considering all possible valid partial $3$-colorings for either the dominating induced $C_6$ or the dominating complete bipartite subgraph.
\section{DIM-less graphs}\label{sec: DIM-less}
		\begin{figure}
	\centering
	\begin{tikzpicture}[scale = 0.6, b/.style={circle, scale=0.4, draw=black, fill},ye/.style={circle, scale=0.4, draw=black, fill=yellow},w/.style={circle, scale=0.4, draw=black},lg/.style={circle, scale=0.4, draw=black, fill=lightgray}, line width=0.5pt]
		\draw[fill=cyan!30] (0,0) ellipse (2cm and 1cm);
		\node (0) [lg] at (0,0.5){};
		\node[above right] at (0,0.15) {\footnotesize $v$};
		\node (1) [lg] at (0,2){};
		\node[above right] at (0,2) {\footnotesize $v_1$};
		\node (2) [lg] at (-1,3){};
		\node[above left] at (-1,3) {\footnotesize $v_2$};
		\node (3) [ye] at (0,4){};
		\node[above right] at (0,4) {\footnotesize $v_3$};
		\node (4) [w] at (0,5){};
		\node (5) [ye] at (0,6){};
		\node (6) [b] at (0,7){};
		\node (7) [b] at (-1,8){};
		\node (8) [b] at (1,8){};
		\node (9) [b] at (0,9){};
		\draw [] (0) -- (1) -- (2) -- (3) -- (1)  (3) -- (4) -- (5) -- (6) -- (7) -- (9) -- (8) -- (6) -- (9) (7) -- (8);
	\end{tikzpicture}
	
	\caption{Graph $G'$: NSF graph $G$ (in cyan color) plus the gadget.}
	\label{fig: np-completeness}
\end{figure}
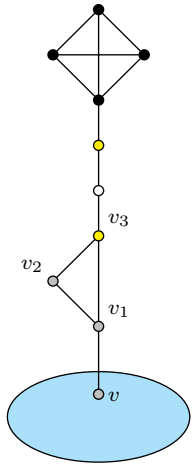

A graph is called a \emph{neighborhood star-free graph} (NSF graph) if every vertex of degree at least two belongs to a triangle. This family of graphs was introduced in~\cite{forte2022new}. In this section, we show that the decision problem PED is NP-complete even when the input is a connected DIM-less graph.

The complexity of the PED problem on NSF graphs was studied in~\cite{forte2022new}, where the following hardness result was obtained.

\begin{lemma}\cite{forte2022new}\label{lem: NFS graphs}
	The problem of deciding whether a given graph has a DIM is NP-complete when the input is a connected NSF graph. Moreover, in the positive instances, all non-trivial PED-sets correspond exactly to the DIMs.
\end{lemma}

Given a connected NSF graph $G$, let $G'$ be the graph obtained by adding a gadget to $G$, a structure with thirteen edges, twelve within the gadget and one connecting a vertex $v$ (chosen as a vertex of maximum degree in $G$) to a vertex $v_1$, as shown in Fig.~\ref{fig: np-completeness}.

We obtain the following result by means of a reduction from the DIM problem. Note that, in the class of NSF graphs, deciding the existence of a non-trivial PED-set is equivalent to deciding the existence of a DIM.
%
\begin{theorem}\label{thm: dimless np completeness}
	Let $G$ be a connected DIM-less graph. Deciding whether $G$ admits a non-trivial PED-set is NP-complete.
\end{theorem}
\begin{proof}
	Let $G$ be a connected NSF graph. We prove the result by applying Lemma~\ref{lem: NFS graphs}.
	
	Note that the six edges in the $K_4$ in the gadget, together with the edge incident to one of its vertices, belong to every PED-set of $G'$. Thus $G'$ is a connected DIM-less graph.
	
	If $v$ has degree at most one, it is easy to prove that $G$ has a DIM and $G'$ has a non-trivial PED-set.
		
	Now, assume that $v$ has degree at least two. 
	
	If $G$ has a DIM, then $v$ is either white or yellow. If $v$ is white (resp. yellow), we can complete the $3$-coloring to $G'$ by using the coloring of Fig.\ref{fig: np-completeness}, changing $v_1$ from gray to yellow (resp. white), and $v_2$ from gray to white (resp. yellow). Clearly, the resulting coloring corresponds to a non-trivial PED-set of $G'$.
	
	Now assume that $F$ is a non-trivial PED-set of $G'$. Notice that if any vertex $u$ in $G$ is black, then each vertex in $G'$ is non-white, since $G$ is an NSF graph, contradicting the fact that $F$ is a non-trivial PED-set. Hence, $v$ must be either white or yellow. Since $v$ belongs to a triangle in $G$, it follows that $E(G)\cap F$ is a DIM of $G$. 
\end{proof}
As a corollary, we obtain the following result, which solves the problem of determining the complexity of deciding whether a graph admits a proper PED-set.
\begin{cor}
	Let $G$ be a graph. Deciding whether $G$ admits a proper PED-set is NP-complete.
\end{cor}
In particular, if $G$ is a DIM-less graph, since every graph admits the trivial PED-set formed by its edge set, we obtain the following immediate corollary:
\begin{cor}\label{cor: DIM-less at least two PEDs}
	Let $G$ be a DIM-less graph. Deciding whether $G$ admits at least two PED-sets is NP-complete.
\end{cor}
Notice that deciding whether $G$ admits at least two PED-sets is NP-complete even when $G$ is an NSF graph. Therefore, the NP-completeness holds for general graphs as well. This result also follows as a consequence of Corollary~\ref{cor: DIM-less at least two PEDs}. 
\section{$P_6$-free graphs}~\label{sec: P_6-free}
In this section, we will state and prove a series of statements that enable us to establish the correctness of the algorithm presented in Section~\ref{sec: the algorithm}, leading to the proof of the following theorem:
\begin{theorem}\label{thm: main result}
	Given a $P_6$-free graph $G$, there is a cubic-time algorithm to find a perfect edge dominating set of minimum size for $G$.
\end{theorem}
Recall that a connected $P_6$-free graph either has a dominating induced $C_6$, or has a dominating (not necessarily induced) complete bipartite subgraph (see Theorem~\ref{p6-free struct}). Based on this observation, we split this section into two parts: in the first one we consider the $P_6$-free connected graphs having a dominating induced $C_6$, and in the second one, we deal with the $P_6$-free connected graphs with a dominating (not necessarily induced) complete bipartite subgraph.
\subsection{Graphs with a dominating induced $C_6$}
Let us begin by considering connected $P_6$-free graphs that have a dominating induced $C_6$. The following lemma outlines the possible valid partial $3$-colorings for such a cycle.  
	\begin{lemma}\label{lem: coloreos del C_6}
		Let $G$ be a graph. All valid partial $3$-colorings $(B,Y,W)$ for an induced $C_6$ in $G$ are depicted in Fig.~\ref{Fig2}.
	\end{lemma}
	\begin{proof}
		Let $G$ be a graph with a valid $3$-coloring $(B,Y,W)$, and let $C=(v_1,v_2,v_3,v_4,v_5,v_6)$ be an induced cycle in $G$. Since a cycle has no pendant vertices, if $v_i\notin W$ for all $1\le i\le 6$, then $v_i\in B$ for each $i$. This corresponds to the valid partial $3$-coloring shown in case (a).
		
		 Assume, without loss of generality, that $v_1\in W$. Thus $v_2,v_6\in Y$. We now consider three cases:
		 \begin{itemize}
		 	\item \textbf{Case 1: } If $v_3\in W$, it implies that $v_4\in Y$, and consequently, $v_5\in B\cup W$, as $v_6\in Y$. If $v_5\in B$, then we obtain the 3-coloring (b). If $v_5\in W$, then we obtain the 3-coloring (e).
		 	\item \textbf{Case 2: } If $v_3 \in Y$, then $v_4 \in W$ and $v_5 \in Y$, corresponding to 3-coloring (d).
		 	\item \textbf{Case 3: } If $v_3 \in B$, then $v_4 \in B$ or $v_4 \in Y$, giving place to the valid partial $3$-coloring (b) or (c) of $C$, respectively.	  
		 \end{itemize}
		Hence, all valid partial $3$-colorings of the induced $C_6$ are as depicted.
	\end{proof}
	\begin{theorem}\label{thm: dominating cycle}
	Let $G$ be a connected $P_6$-free graph, and let $C$ be a dominating induced $C_6$ in $G$. Then $|N_{C}(v)|\ge 2$ for every $v\in V(G)\setminus V(C)$. Moreover, if $(B,Y,W)$ is a valid $3$-coloring of $G$, then one of the following conditions holds:
	\begin{itemize}
		\item The corresponding partial $3$-coloring of $C$ is of type (a), with $B=V(G)$ and $Y=W=\emptyset$. This valid 3-coloring corresponds to the trivial PED-set.
		\item The corresponding partial $3$-coloring of $C$ is of type (b), $c$ is the unique black vertex of $C$ and $d\in V(C)$ is the unique yellow vertex with only white neighbors in $V(C)$, then:
		\begin{itemize}
			\item $W\setminus V(C)=\{v\in V(G)\setminus V(C):N_C(v)\subseteq Y\}$,
			\item $Y\setminus V(C)=\{v\in V(G)\setminus V(C):\, (N_C(v)\subseteq \{c,d\}\cup W) \wedge (|N_C(v)\cap \{c,d\}|\leq 1)\}$,
			\item Exactly one of the following cases holds:
			\begin{itemize}
			      \item $B \setminus V(C) = \emptyset$ if $|(Y\setminus V(C))\cap N(d)|=1$.
			      \item $B \setminus V(C) = \{x\}$, $N_C(x)=\{c,d\}$ and $(Y\setminus V(C))\cap N(d)=\emptyset$.
			\end{itemize}

		\end{itemize}    
		\item The corresponding partial $3$-coloring of $C$ is of type (c) with: 
		\begin{itemize}
			\item $B\setminus V(C)=\{v\in V(G)\setminus V(C):\,N_C(v)\subseteq B\}$,
			\item $Y\setminus V(C)=\{v\in V(G)\setminus V(C):|N_C(v)|=2
			\wedge |N_C(v)\cap W|=|N_C(v)\cap B|=1\}$,
			\item  $W\setminus V(C)=\{v\in V(G)\setminus V(C):N_C(v)\subseteq Y\}$.
		\end{itemize}
		\item The corresponding partial $3$-coloring of $C$ is of 
		type (d) with: 
		\begin{itemize}
			\item $B=\emptyset$,
			\item $Y\setminus V(C)=\{v\in V(G)\setminus V(C): N_C(v)\subseteq W\}$,
			\item $W\setminus V(C)=\{v\in V(G)\setminus V(C):N_C(v) \subseteq Y\}$.
		\end{itemize}
		Thus, this valid 3-coloring corresponds to a DIM.
		\item The corresponding partial $3$-coloring of $C$ is of 
		type (e) with: 
		\begin{itemize}
			\item $Y\setminus V(C)=\{v\in V(G)\setminus V(C):|N_C(v)\cap Y|\le 1\}$, which is a dissociation set,
			\item $(B\cup W)\setminus V(C)=\{v\in V(G)\setminus V(C):\,N_C(v)\subseteq Y\}$ and $B\cup W$ is an independent set,
			\item If $G[Y]$ is 1-regular, then $B=\emptyset$ (the valid 3-coloring corresponds to a DIM). Otherwise, $B=\{x\}$ and $N(x)=\{y\in Y:y$ has degree zero in $G[Y]\}$.
		\end{itemize}
	\end{itemize}
	Moreover, every valid partial 3-coloring of $C$ can be extended to at most one valid 3-coloring except those of type $(e)$, in which case the number of PED-sets is $O(n)$.
\end{theorem}

\begin{proof}
	Assume that $C=(v_1,v_2,v_3,v_4,v_5,v_6)$ is a dominating induced $C_6$ of $G$. Suppose that there exists a vertex $v\in V(G)\setminus V(C)$ with exactly one neighbor in $V(C)$, say $v_1$. Then the set $\{v,v_1,v_2,v_3,v_4,v_5\}$ induces $P_6$ in $G$, contradicting the $P_6$-freeness of $G$. Hence, every $v\in V(G)\setminus V(C)$ satisfies $|N_C(v)|\ge 2$. 
	
	Now, consider each valid partial $3$-coloring type of $C$:
	\vspace{0.3cm}

	\textbf{Type (a):} Since each vertex in the dominating set $V(C)$ is black, it follows that each vertex in $V(G)$ must be black.
	 	
	\vspace{0.3cm}
	\textbf{Type (b):} Let $v\in V(G)\setminus V(C)$. 
	\begin{itemize}
		\item If $v\in W\setminus V(C)$, then $N_C(v)\subseteq Y$. Conversely, if $N_C(v)\subseteq Y$, noting that two of the three yellow vertices in $V(C)$ have non-white neighbors, then $v\in W\setminus V(C)$. 
		\item If $v \in Y$, then $N_C(v)$ intersects $W\cap V(C)$, $B\cap V(C) = \{d\}$, or $Y\cap V(C) = \{c\}$. Thus, $N_C(v) \cap W \cap V(C)\neq \emptyset$ and $|N_C(v)\cap \{c,d\}|\leq 1$. Conversely, if $N_C(v) \cap W \cap V(C)\neq \emptyset$, then $v \in Y$. On the one hand, if $B \setminus V(C) = \emptyset$, then $|(Y\setminus V(C))\cap N(d)|=1$. On the other hand, if $x\in B \setminus V(C)\neq \emptyset$, then $N_C(x)=\{c,d\}$ and thus $B \setminus V(C) = \{x\}$.
	\end{itemize}
	
	\textbf{Type (c):}
	\begin{itemize}
		\item If $v \in W\setminus V(C)$, then $N_C(v) \subseteq Y$. Conversely, if $N_C(v) \subseteq Y$ and each yellow vertex in $V(C)$ has a non-white neighbor, then $v \in W$.
		\item If $v \in Y$, then $|N_C(v)| = 2$, and $|N_C(v) \cap W| = |N_C(v) \cap B| = 1$, as $v$ must connect to both a white and a black vertex. Conversely, if $|N_C(v)|=2$, and $|N_C(v)\cap W|=|N_C(v)\cap B|=1$, then  $v\in Y\setminus V(C)$.
		\item If $v \in B\setminus V(C)$, then $N_C(v) \subseteq B$, since yellow vertices in $V(C)$ already have a non-white neighbor. Conversely, if $N_C(v) \subseteq B$, then $v \in B$.
	\end{itemize} 

	\textbf{Type (d):}
	\begin{itemize}
		\item If $v \in W\setminus V(C)$, then $N_C(v) \subseteq Y$. Conversely, if $N_C(v) \subseteq Y$, and each yellow vertex in $V(C)$ has a non-white neighbor, then $v \in W\setminus V(C)$.
		\item If $v \in Y\setminus V(C)$, then $N_C(v) \subseteq W$, as all non-white neighbors of yellow vertices in $V(C)$ lie within $V(C)$. Conversely, if $N_C(v)\subseteq W$, then  $v\in Y\setminus V(C)$.
		\item If all non-white neighbors of yellow vertices in $V(C)$ are in $V(C)$, then $B\setminus V(C)=\emptyset$.
	\end{itemize}
	 
	\textbf{Type (e):}
	\begin{itemize}
		\item If $v \in Y$ and has no non-white neighbors in $V(C)$, then $N_C(v) \subseteq W$. Otherwise, $|N_C(v) \cap Y| = 1$. Conversely, if $N_C(v) \subseteq W$ or $|N_C(v) \cap Y| = 1$, and $|N_C(v)| \ge 2$, then $v \in Y\setminus V(C)$. Clearly, $Y\setminus V(C)$ is a dissociation set.
		\item Clearly, if $N_C(v) \subseteq Y$ and $|N_C(v)| \ge 2$, then $v \in B \cup W$. Additionally, $|B| = 1$ and thus $(B\cup W)\setminus V(C)$ is an independent set.
		
		The last condition immediately follows.
		%
		%
	\end{itemize}
\end{proof}

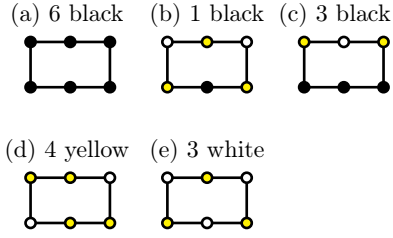
\begin{figure}
	\centering
	\begin{tikzpicture}[scale = 0.6, b/.style={circle, scale=0.4, draw=black, fill},ye/.style={circle, scale=0.4, draw=black, fill=yellow},w/.style={circle, scale=0.4, draw=black},
		line width=1pt]
		\node (0) [b] at (-0.87,0){};
		\node (1) [b] at (0,0){};
		\node (2) [b] at (0.87,0){};
		\node (3) [b] at (0.87,1){};
		\node (4) [b] at (0,1){};
		\node (5) [b] at (-0.87,1){};
		\node (a)[scale=0.9] at (0-0.075,1.6){(a) 6 black};
		\draw[] (0) -- (1) -- (2) -- (3) -- (4) -- (5) -- (0);
		\node (a0) [ye] at (-0.87+3,0){};
		\node (a1) [b] at (0+3,0){};
		\node (a2) [ye] at (0.87+3,0){};
		\node (a3) [w] at (0.87+3,1){};
		\node (a4) [ye] at (0+3,1){};
		\node (a5) [w] at (-0.87+3,1){};
		\node (a)[scale=0.9] at (3,1.6){(b) 1 black};
		\draw[] (a0) -- (a1) -- (a2) -- (a3) -- (a4) -- (a5) -- (a0);
		\node (b0) [b] at (-0.87+6,0){};
		\node (b1) [b] at (0+6,0){};
		\node (b2) [b] at (0.87+6,0){};
		\node (b3) [ye] at (0.87+6,1){};
		\node (b4) [w] at (0+6,1){};
		\node (b5) [ye] at (-0.87+6,1){};
		\node (a)[scale=0.9] at (5.8,1.6){(c) 3 black};
		\draw[] (b0) -- (b1) -- (b2) -- (b3) -- (b4) -- (b5) -- (b0);
		
		\node (0) [w] at (-0.87,-3){};
		\node (1) [ye] at (0,-3){};
		\node (2) [ye] at (0.87,-3){};
		\node (3) [w] at (0.87,-2){};
		\node (4) [ye] at (0,-2){};
		\node (5) [ye] at (-0.87,-2){};
		\node (a)[scale=0.9] at (0-0.075,1.6-3){(d) 4 yellow};
		\draw[] (0) -- (1) -- (2) -- (3) -- (4) -- (5) -- (0);
		\node (a0) [ye] at (-0.87+3,-3){};
		\node (a1) [w] at (0+3,-3){};
		\node (a2) [ye] at (0.87+3,-3){};
		\node (a3) [w] at (0.87+3,-2){};
		\node (a4) [ye] at (0+3,-2){};
		\node (a5) [w] at (-0.87+3,-2){};
		\node (a)[scale=0.9] at (3,1.6-3){(e) 3 white};
		\draw[] (a0) -- (a1) -- (a2) -- (a3) -- (a4) -- (a5) -- (a0);
	\end{tikzpicture}
	
	\caption{Five possible valid partial $3$-colorings of an induced dominating $C_6$}
	\label{Fig2}
\end{figure}
\subsection{Graphs with a dominating complete bipartite subgraph}
Now, let us examine a connected $P_6$-free graph $G$ that contains a dominating, though not necessarily induced, complete bipartite subgraph $K$, where $K$ includes $K_{p,q}$ as a spanning subgraph with a bipartition $\{R,S\}$, where $|R|=p$ and $|S|=q$. We will explore all possible valid partial $3$-colorings for $K$. It is no hard to see that any dominating complete bipartite subgraph can be extended to a maximal one in linear-time. Hence, we can assume that $K$ is maximal, assuming this additional cost.
\subsubsection{$V(K)\cap B=\emptyset$}\label{subsubsec: V(K) int B empty}
First, we will consider the case in which $K$ has a valid partial $3$-coloring with no black vertices. Let us start by considering the case in which $G$ is also a bipartite graph.

	\begin{lemma}\label{lem: induced bipartite w-y}
		Let $G$ be a bipartite connected $P_6$-free graph with at least two vertices, and let $\{M,N\}$ be its bipartition. If $G$ admits a valid $3$-coloring $(B,Y,W)$ such that $N\subseteq Y$, then there exists a vertex in $M$ that is adjacent to every vertex in $N$.
	\end{lemma}
	\begin{proof}
		By Theorem~\ref{p6-free struct}, the graph $G$ either contains a dominating induced $C_6$ or a dominating complete bipartite subgraph $K_{p,q}$. 

		Assume first that $G$ has a dominating induced cycle $C=(v_1,v_2,v_3,v_4,v_5, v_6)$. Without loss of generality, suppose the vertices with odd indices lie in $M$. Since $G$ is $P_6$-free, each vertex in $M$ must have least two neighbors in $\{v_2,v_4,v_6\}$. Hence, every vertex in $M$ is either white or black, as yellow vertices have exactly one non-white neighbor. Besides, since $G$ is bipartite and $N\subseteq Y$, at least one vertex in $M$ must be black. Furthermore, $|M\cap B|=1$; otherwise, two black vertices would share a neighbor in $\{v_2,v_4,v_6\}$, violating the validity of the coloring. Thus, every vertex $u\in N$ must be adjacent to the unique black vertex in $M$, since $u$ needs a non-white neighbor and $N(u)\subseteq M$.
				
		Now, suppose $G$ contains a dominating complete bipartite subgraph $K$ isomorphic to $K_{p,q}$. Without loss of generality, assume that $p+q$ is as large as possible. Let $\{U,V\}$ be the bipartition of $K$, where $U\subseteq M$, $V\subseteq N$, $|U|=p$, and $|V|=q$. The lemma is trivially true if $p=1$.
				
		Assume now that $p\ge 2$. Suppose that $(B,Y,W)$ is a valid $3$-coloring of $G$, and let $u\in U$ be a non-white vertex. Then, there can be no other non-white vertex in $M$; otherwise, some vertex in $V$ would be adjacent to more than one non-white neighbor, contradicting the properties of yellow vertices. Therefore, $u$ must be adjacent to every vertex in $N$.
					
		Now, assume $U \subseteq W$, so there must exist a non-white vertex $c_1 \in M \setminus U$. Since $p+q$ is maximum, $c_1$ is adjacent to some $v_1 \in V$ and non-adjacent to at least one other vertex $v_2 \in V$. Additionally, there exists another non-white vertex $c_2 \neq c_1$ in $M \setminus U$ that is adjacent to $v_2$. Because the coloring is valid, $c_2$ must be non-adjacent to $v_1$.
		
		Suppose, towards a contradiction, that no vertex in $M$ is adjacent to all vertices in $N$. Hence, there exists a vertex $d_1\in N\setminus V$. Since $d_1\in Y$, it can be adjacent to at most one of $c_1$ and $c_2$. Suppose, without loss of generality, that $d_1$ is adjacent only to $c_1$. Since $p+q$ is maximum, there exists $u\in U$ such that $u$ is non-adjacent to $d_1$. Consequently, the set $\{d_1,c_1,v_1,u,v_2,c_2\}$ induces $P_6$, contradicting the assumption that $G$ is $P_6$-free. 
		
		Suppose now that $d_1$ is non-adjacent to both $c_1$ and $c_2$. Thus, there must exist a third non-white vertex $c_3\in M\setminus U$, distinct from both $c_1$ and $c_2$, that is adjacent to $d_1$. Besides, since $K$ is a dominating complete bipartite subgraph and the coloring is a valid, $c_3$ must be adjacent to some $v_3\in V$, distinct from $v_1$ and $v_2$. Using a similar argument as before (replacing $c_1$ with $c_3$, and $v_1$ with $v_3$), we again obtain an induced $P_6$, a contradiction. 
		
		In all cases, assuming that no vertex in $M$ is adjacent to every vertex in $N$ leads to a contradiction. Therefore, such a vertex must exist.		
	\end{proof}
	The following lemma describes a possible valid partial $3$-coloring of $K$ when $V(K)\cap B=\emptyset$. Recall that $\{R,\, S\}$ denotes the bipartition of $K$.
		\begin{lemma}\label{lem: dominating star}
			Let $G$ be a $P_6$-free graph that contains a dominating (not necessarily induced) complete bipartite subgraph $K$. If $G$ admits a valid $3$-coloring $(B,Y,W)$ such that $V(K)\cap B=\emptyset$, then exactly one of the following conditions holds:
			\begin{enumerate}
				\item $|R\cap Y|=|S\cap Y|=1$, and $K$ contains $K_{1,t}$ as a spanning tree, for some positive integer $t$; or
				\item $R\subseteq W$ and $S\subseteq Y$, or $R\subseteq Y$ and $S\subseteq W$.
			\end{enumerate}
		\end{lemma}
		\begin{proof}
			If condition 2. holds, then there is nothing to prove. Suppose instead that condition $2.$ does not hold, and let $r\in R$ and $s\in S$ be yellow vertices that are adjacent vertices. Thus, all vertices in $R\setminus\{r\}$ and $S\setminus\{s\}$ must be white. This implies that either $|R|=1$ or $|S|=1$, since otherwise there would exist two adjacent white vertices. Consequently, $K$ contains a spanning tree isomorphic to $K_{1,t}$ for some positive integer $t$.
		\end{proof}
		The following lemma describes how all possible valid partial $3$-coloring for $K$ can be extended to a valid $3$-coloring for $G$ when $|R\cap Y|=|S\cap Y|=1$.
		\begin{lemma}\label{lem: valid 3-coloring}
			Let $G$ be a $P_6$-free graph containing a maximal dominating (not necessarily induced) complete bipartite subgraph $K$. If $G$ has a valid $3$-coloring $(B,Y,W)$ such that $|R\cap Y|=|S\cap Y|=1$, then:
			\begin{enumerate}
			\item $|R|=1$ or $|S|=1$. 
			\item If $R = \{r\}$, then $S=N_G(r)$, and $G[S]$ verifies exactly one of the three following conditions:
			\begin{itemize}
				\item $G[S]$ does not have any edge, which implies that $S$ is an independent set.
				\begin{enumerate}
					\item Let $Y'=\{x\in V(G)\setminus V(K): |N^S(x)|\geq 2\}$, and $W'=\{s\in S: N^{Y'}(s)\ne\emptyset\}$. Then $Y'\subseteq Y$ and $W'\subseteq W$.
					\item Let $G'=G-(R\cup Y'\cup W')$. Any pair of vertices $x,x' \in S\setminus W'$ verifies $N_{G'}(x)\cap N_{G'}(x')=\emptyset$. Let $W''=\{s\in S\setminus W': N_{G'}(s) \mbox{ contains two adjacent vertices}\}$, and let $Y''=\{x\in V(G)\setminus (V(K)\cup Y'): N^{W''}(x)\ne\emptyset\}$. Then $S'=S\setminus (W'\cup W'')\ne\emptyset$ for some vertex $s\in S'\cap Y$,  $V(G)\setminus (V(K)\cup Y'\cup Y''\cup N(s))\cup\{s\}\subseteq Y$, and $(V(G)\setminus(V(K)\cup Y'\cup Y'')\cap N(s))\cup (S'\setminus \{s\})\subseteq W$. In addition, $V(G)\setminus (V(K)\cup N(s))$ induces a $1$-regular subgraph in $G$.
				\end{enumerate}
				\item $G[S]$ has exactly one edge $ss'$. Then one of $s$ or $s'$ must be the unique yellow neighbor of $r$, and the color of all other vertices can be determined. There are at most two valid 3-colorings.
				\item $G[S]$ has exactly $k\geq 2$ edges, and there is a unique vertex $s$ of degree $k$ in $G[S]$. Then $s$ is the unique yellow neighbor of $r$, and the color of all other vertices can be determined. There is at most one valid 3-coloring.
			\end{itemize}
			\item Any of these valid 3-coloring corresponds to a DIM and there are at most $O(n)$ different DIMs.
			\end{enumerate}			
		\end{lemma}
		\begin{proof}
			Since $|R\cap Y|=|S\cap Y|=1$, we have $R\cap Y=\{r\}$ and $S\cap Y=\{s\}$. By Lemma \ref{lem: dominating star}, the subgraph $K$ contains a spanning tree isomorphic to $K_{1,t}$ for some positive integer $t$. Therefore, $|R|=1$ or $|S|=1$, and Condition 1 follows.
			
			Assume that $R = \{r\}$. Since $K$ is maximal, then $S = N_G(r)$. Since $V(K)\cap B=\emptyset$, every vertex in $S\setminus\{s\}$ is colored white.
			
			Suppose that $S$ is an independent set. Let $x\in V(G)\setminus V(K)$. If $|N^S(x)|\geq 2$, then $x\in Y$, because $x$ has at least one white neighbor in $S$. Thus, $s'\in W$ for each vertex $s'\in W'$, and $N_{G'}(x)\cap N_{G'}(x') = \emptyset$ for each $x,\, x'\in S\setminus W'$. Analogously, it can be proved that $W''\subseteq W$ and $Y''\subseteq Y$. Therefore, $S'=S\setminus (W'\cup W'')\ne\emptyset$ and some of $s\in S'$ must be a yellow vertex.
			
			Let $x\in W''$ with two adjacent neighbors $y$ and $z$ in $V(G)\setminus (V(K)\cup Y')$. Suppose that $x$ is a yellow vertex. Thus, $y$ must be yellow and $z$ must be white, or vice versa. Additionally, $x$ is a yellow vertex with two non-white neighbors, $y$ (or $z$) and $r$, contradicting that the $3$-coloring is admissible. Therefore, $W''\subseteq W$ and $Y''\subseteq Y$. The rest of this part of the statement follows by the definition of admissible $3$-coloring.
			
			Suppose that $E(G[S]) = \{ss'\}$. Since $s$ and $s'$ are adjacent to the yellow vertex $r$, one of them must be a yellow vertex. The color of the vertices in $V(G)\setminus V(K)$ depends on they are adjacent or non-adjacent to this yellow vertices. Therefore, there are at most two valid total $3$-colorings.
			
			Suppose now that $|E(G[S])|\ge 2$. Since $r$ is a yellow vertex, $G[S]$ is isomorphic to $K_{1, g} + h K_1$ with $g + h = |S|$. Additionally, the vertex $x$ of maximum degree in $G[S]$ must be the yellow vertex of $S$. Consequently, there is only one valid $3$-coloring.
			
			Finally, it is clear that $B=\emptyset$. Therefore, each valid total $3$-coloring is a DIM.
		\end{proof}
		Given a connected $P_6$-free graph $G$ with a dominating (not necessarily induced) complete bipartite subgraph $K$, and a valid $3$-coloring $(B,Y,W)$ such that $R\subseteq W$ and $S\subseteq Y$, we use $X$ to denote the subset of vertices in $V(G)\setminus V(K)$ with at least one neighbor in $R$. We denote the subgraph $G-(V(K)\cup X)$ by $G_0$. Note that $S$ dominates $V(G_0)$. Let $C$ be a connected component of $G_0$. When $C$ is bipartite, we denote the bipartition of $C$ by $\{U_C,V_C\}$.
	\begin{lemma}\label{lem: K white and yellow}
		Let $G$ be a connected $P_6$-free graph that contains a dominating (not necessarily induced) complete bipartite subgraph $K$. If $G$ admits a valid $3$-coloring $(B,Y,W)$ such that $R\subseteq W$ and $S\subseteq Y$, then the following conditions hold:
	\begin{enumerate}
		\item $X\subseteq Y$.
		\item $X\cup S$ is a dissociation set.
		\item If $C$ is a non-bipartite connected component of $G_0$, then $V(C)\subseteq B$.
		\item If $C$ is a bipartite connected component of $G_0$, then one of the following conditions holds:
		\begin{enumerate}
			\item $V(C)\subseteq B$; or
			\item $U_C\subseteq W$ and $V_C\subseteq Y$; or
			\item $U_C\subseteq Y$ and $V_C\subseteq W$.
		\end{enumerate}
	\end{enumerate}
	\end{lemma}
	\begin{proof}
		Let $\pi$ be a valid $3$-coloring assignment of $G$. 		
		
		Statement $1.$ follows from the definition of a valid $3$-coloring and the fact that $R\subseteq W$.
		
		Clearly, Statement $2.$ holds.
		
		Let $C$ be a connected component of $G_0$. If $C$ contains a black vertex $u$, then, since $S$ dominates $V(C)$, it follows by induction on the distance from $u$ that every vertex in $C$ must also be black.
		
		Now suppose $C$ is a non-bipartite component of $G_0$, and let $H$ be an odd induced cycle in $C$. If no vertex of $H$ is black, then $H$ must contain two adjacent yellow vertices $u$ and $v$. But then $u$ would have two non-white neighbors: $v$ and some $x \in S$ that dominates $u$, contradicting the coloring condition. Hence, all vertices of $H$, and therefore of $C$, must be black.
		
		Now suppose $C$ is a bipartite component of $G_0$. If $V(C) \subseteq B$, then condition (a) holds. Assume instead that $C$ contains non-black vertices. Since $S$ dominates $V(C)$ and no yellow vertex can have two non-white neighbors, it follows that all vertices in $C$ are either white or yellow, and no two adjacent vertices can be yellow. Thus, the coloring induces a proper 2-coloring on $C$ with colors white and yellow. Since the bipartition of a connected bipartite graph is unique (up to swapping sides), we conclude that either $U_C \subseteq W$ and $V_C \subseteq Y$, or $U_C \subseteq Y$ and $V_C \subseteq W$.
	\end{proof}
	By $S_1$ (resp. $X_1)$, we refer to the subset of vertices $v\in S$ (resp. $v\in X)$ such that $d_{G[X\cup S]}(v)=1$. Let $S_2=S\setminus S_1$ (resp. $X_2=X\setminus X_1$).
	Note that $X_2\cup S_2$ is an independent set, since $d_{G[X\cup S]}(v)=0$ for all $v\in X_2\cup S_2$. Let $C$ be a non-trivial bipartite  connected component of $G_0$ that has at least one vertex $x$ with a neighbor in $S_2$. We say that $C$ is of \emph{type I} if there exists a vertex $v\in S_2$ such that $v$ has a neighbor $c$ and a non-neighbor $d$, both in $V(C)$, and $cd$ is an edge of $E(C)$. Such vertex $v$ is called an~\emph{indicator} of $C$. The component $C$ is called of \emph{type II} if $|V(C)|\ge 2$ and every vertex in $S_2$ with a neighbor in this component is adjacent to each vertex of $V(C)$.
	\begin{rmk}\label{rmk: invalid coloring}
		Let $G$ be a connected $P_6$-free graph having a dominating (not necessarily induced) complete bipartite subgraph $K$ with $R\subseteq W$ and $S\subseteq Y$. If $N_G(v)\cap V(G_0)=\emptyset$ for some $v\in X_2\cup S_2$, then $G$ has no valid $3$-coloring.
	\end{rmk}
	\begin{proof}
		Suppose, towards a contradiction, that $G$ has a valid $3$-coloring, and $v$ is a vertex in $X_2\cup S_2$ such that $N_G(v)\cap V(G_0)=\emptyset$. If $u\in N_G(v)$, then $u\not\in V(G_0)$; i.e., $u\in V(K)\cup X=R\cup S\cup X$. Furthermore, $u\not\in X\cup S$ because, according to the definition of $X_2$ and $S_2$, it follows that $d_{G[X\cup S]}(v)=0$. Hence, $u\in R$. Since $R\subseteq W$, $u\in W$. Thus, as $X_2\cup S_2\subseteq Y$, $v$ is a yellow vertex with all its neighbors being white vertices. Therefore, $G$ has no valid 3-coloring, contradicting the assumption that $G$ has a valid $3$-coloring.
	\end{proof}
	The following lemma ensures that the colors assigned to the vertices in a connected component of type I within a valid $3$-coloring for $G$ also determine the colors of each vertex in any other component of type I.
	\begin{lemma}\label{lem: components of type I}
		Let $G$ be a connected $P_6$-free graph with a dominating (not necessarily induced) complete bipartite subgraph $K$ where $R\subseteq W$ and $S\subseteq Y$. Let $C$ be a connected component of type I in $G_0$, and let $v$ be an indicator of $C$. Suppose $c$ and $d$ are two adjacent vertices in $C$ such that $v$ is adjacent to $c$ and non-adjacent to $d$. If $x\in (X_2\cup S_2)\setminus\{v\}$, then one of the following conditions holds:
		\begin{enumerate}
			\item $x$ is adjacent to $c$ or $d$, or
			\item $N^{V(G_0)\setminus V(C)}(x)\subseteq N^{V(G_0)\setminus V(C)}(v)$.
		\end{enumerate}
	\end{lemma}
	\begin{proof}
		If condition 1. holds, there is nothing to prove. Let us assume then that $x\in X_2\cup S_2$, with $x\neq v$, and it is non-adjacent to both $c$ and $d$. Suppose, towards a contradiction, that there exists a vertex $z\in V(G_0)\setminus V(C)$ that is adjacent to $x$ and non-adjacent to $v$. Recall that $v$ is non-adjacent to $x$ (due to the definition of $X_2$ and $S_2$). Besides, there exists a vertex $r\in R$ that is adjacent to both $x$ and $v$ (due to the completeness of $K$). Hence, $\{z,x,r,v,c,d\}$ induces $P_6$. The contradiction arises from assuming the existence of a vertex $z$ adjacent to $x$ and non-adjacent to $v$. Therefore, every vertex $z\in V(G_0)\setminus V(C)$ adjacent to $x$ is also adjacent to $v$.
		
	\end{proof}
	By using similar arguments to Lemma~\ref{lem: components of type I}, we obtain the following remark.
	\begin{rmk}\label{rmk: components of type I}
		Let $C$ and $C'$ be connected components of type I. Assume that $C$ has an indicator $v$, that is adjacent to $c$ and non-adjacent to $d$ with $cd\in E(C)$ such that: 
		\begin{enumerate}
		        \item there exists a vertex $w\in S_2$ adjacent to $d$ and non-adjacent to $c$, or
		        \item $V(C)\subseteq B$.
		\end{enumerate}
		Then the coloring of $C$ determines the coloring of $C'$.
		\end{rmk}
	The below lemma characterizes the component of type II of $G_0$ when $G$ has a valid $3$-coloring.
	\begin{lemma}\label{lem: type II}
		Let $G$ be a connected $P_6$-free graph with a dominating (not necessarily induced) complete bipartite subgraph $K$ where $R\subseteq W$ and $S\subseteq Y$. If $G$ has a valid $3$-coloring, and $C$ is a connected component of $G_0$ of type II, then either $|U_C|=1$, or $|V_C|=1$.
	\end{lemma}
	\begin{proof}
		Suppose, towards a contradiction, that $|U_C|\ge 2$ and $|V_C|\ge 2$. Since $C$ is a connected component of $G_0$ of type II, there exists a vertex $v\in S_2\subseteq S$ adjacent to every vertex in $C$. Since $v$ is a yellow vertex, it has at most one non-white neighbor in $C$. As a result, $C$ has two adjacent white vertices, leading to a contradiction. Therefore, $|U_C|=1$ or $|V_C|=1$.
	\end{proof}
	Let $Z$ denote the set of vertices with degree zero in $G_0$ and no neighbor in $S_1\cup X_1$. Suppose $\pi$ is a valid $3$-coloring of $G$, and let $S_2^\pi$ (resp. $X_2^\pi$) be the subset of vertices $v$ in $S_2$ (resp. $X_2$) such that $\pi(u)=w$ for every $u\in N_{G_0}(v)$ with $d_{G_0}(u)\ge 1$. We define $H_\pi$ as the induced bipartite subgraph $G[S_2^\pi \cup X_2^\pi \cup Z]$ of $G$. The following lemma addresses the trivial connected components of $G_0$.
	\begin{lemma}\label{lem: optimum coloring}
		Let $G$ be a connected $P_6$-free graph that contains a dominating (not necessarily induced) complete bipartite subgraph $K$, with $R\subseteq W$ and $S\subseteq Y$. Then, $G$ admits a valid $3$-coloring $\pi$ such that each connected component $C$ of $H_\pi$ has a valid $3$-coloring in which the vertices of $V(C)\cap S_2^\pi$ are assigned the yellow color. Moreover:
		\begin{enumerate}
			\item Such a $3$-coloring of $C$ is valid if and only if $V(C)\cap (S_2^\pi\cup X_2^\pi)$ can be partitioned into neighborhoods of non-white vertices in $V(C)\cap Z$, and for each $z\in V(C)\cap Z$, the vertex $z$ is yellow if and only if $|N_C(z)|=1$. Hence, any PED-set of $C$ contains exactly $|V(C)\cap (S_2^\pi\cup X_2^\pi)|$ edges.
			\item $\pi$ corresponds to a DIM if and only if $X_2^\pi=\emptyset$ and for every vertex $s\in S_2^\pi$, its unique non-white neighbor $z_s$ has degree 1 in $H_\pi$.
			\item Some vertex $z$ of $V(C)\cap Z$ has degree $|V(C)\cap (S_2^\pi\cup X_2^\pi)|$ in $C$. A valid $3$-coloring can be obtained by assigning the white color to all vertices of $V(C)\cap Z$ except $z$.
			\item Other valid $3$-colorings can be obtained by assigning the white color to $z$ and considering the valid 3-colorings of $C - \{z\}$ (note that $C - \{z\}$ may be disconnected). 
		\end{enumerate}
	\end{lemma}	
	\begin{proof}
		The necessary and sufficient conditions follow immediately, as do Statements 1 and 2. It is worth mentioning that, if $X_2^\pi \neq \emptyset$, then there exists a vertex $z \in Z$ with $\pi(z) = b$, since $z$ has at least one neighbor in $S$.
		
		Let $\pi$ be a valid $3$-coloring for $G$. Hence, $H_\pi$  has a valid $3$-coloring $\pi'$ defined as $\pi$ restricted to $V(H_\pi)$; i.e, $\pi'(v)=\pi(v)$ for every vertex $v\in V(H_\pi$). By Lemma~\ref{lem: induced bipartite w-y}, $H_\pi$ has a vertex $z^*\in Z$ adjacent to each vertex of $S_2^\pi$. Therefore, if $\pi$ is an optimum valid $3$-coloring of $G$, $\pi^*$ defined as follows: $\pi^*(u)=\pi(u)$, for each $u\in G-Z$, $\pi^*(z^*)=b$ whenever $|N(z^*)\cap S_2'|\ge 2$, $\pi^*(z^*)=y$ whenever $|N(z^*)\cap S_2'|= 1$, and $\pi^*(z)=w$ for every $z\in Z\setminus\{z^*\}$ turns into a valid $3$-coloring for $G$. 
		
		Statement 4. follows from the fact that if for a vertex $v$ is assigned white in valid $3$-coloring, the $3$-coloring $\pi$ restricted to $V(G-\{v\})$ is also valid.
	\end{proof}
	\subsubsection{$R\cap B\neq\emptyset$ and $S\cap B=\emptyset$}
	Now, we will examine $P_6$-free graphs having a valid $3$-coloring and a dominating (not necessarily induced) complete bipartite subgraph $K$ with at least one black vertex.
	
	We begin by analyzing the valid partial $3$-coloring of $K$ in which exactly one of the sets of the bipartition contains black vertices. 
	
	Let us assume, without losing generality, that $R\cap B\neq\emptyset$. If $|R\cap B|\ge 2$, then there exists a black vertex in $S$, contradicting the fact that $S\cap B=\emptyset$. Therefore, $|R\cap B|=1$. On the other hand, $S\cap W=\emptyset$, since $R\cap B\neq\emptyset$ and the black vertices have no white neighbors. Hence, $S\subseteq Y$.
	\begin{lemma}\label{lem: one black vertex}
		Let $G$ be a graph with a valid $3$-coloring $(B,Y,W)$. Suppose that $G$ has a dominating (not necessarily induced) complete bipartite subgraph $K$ with a bipartition $\{R,S\}$ such that $R\cap B=\{r\}$ and $S\subseteq Y$. Then the following conditions hold:
		\begin{enumerate}
			\item $S$ is an independent set.
			\item $R\setminus\{r\}\subseteq W$, and $R$ is an independent set. 
			\item Let $M$ be the set of vertices in $G-V(K)$ that have at least one neighbor in $S$. Then $M\subseteq W$, and $R\cup M$ is an independent set.
			\item 
			Let $N=V(G)\setminus(V(K)\cup M)$, and define $d^N(x)$ as the degree of a vertex $x$ in the induced subgraph $G[N]$. Then $N\subseteq (B\cup Y)$ and: 
			\begin{enumerate}
				\item If $d^N(x)\ge 2$, then $x\in B$ and $x$ has exactly one neighbor in $K$  which is $r$.
				\item If $d^N(x)=0$, then $x\in Y$ and $x$ has $r$ as neighbor.
				\item If $d^N(x)=1$ and $x$ has at least two neighbors in $R\cup M$, then $x\in Y$ and its neighbors in $R$ are white vertices; i.e., $x$ is non-adjacent to $r$.
			\end{enumerate}				
		\end{enumerate}
	\end{lemma}
	\begin{proof}
	  The first three statements follow directly from the definition of a valid $3$-coloring.
	  	  
	  Since each vertex $x\in N$ is adjacent to at least one vertex in $R$, it follows from Statement~2 that $x$ is a non-white vertex. If $d^N(x)\ge 2$, then $x\in B$, and $N_K(x)=\{r\}$. If $d^N(x) = 0$, since $N_G(x)\cap V(K)\subseteq R$, then $x\in Y$ and must be adjacent to $r$. Now suppose that  $d^N(x)=1$ and $x$ has at least two neighbors in $R\cup M$. Then $x$ has at least one white neighbor in $R\cup M$, and one non-white neighbor in $N$, which implies that $x$ must be a yellow. In particular, $x$ cannot be adjacent to $r$.
	\end{proof}
	\begin{rmk}
	        It is always advisable to choose the one that has the lowest degree in $G$ to obtain a smaller PED-set. Once $r$ is chosen, the valid 3-coloring can be completed as follows: the black vertices of $N$ are exactly neighbors of $r$ and have degree at least one in $G[N]$.
	\end{rmk}
	\subsubsection{$R\cap B\neq\emptyset$ and $S\cap B\neq\emptyset$}
	Now, let us consider $P_6$-free graphs having a dominating complete bipartite subgraph $K$ with a partial $3$-coloring where both sets of the bipartition contain black vertices. It is important to note that in this case $V(K)$ does not have any white vertices.
	\begin{lemma}\label{lem: star - black vertices}
		Let $G$ be a connected $P_6$-free graph with a valid $3$-coloring $(B,Y,W)$. If $G$ has a dominating complete bipartite subgraph $K$ such that $R\cap B\neq\emptyset$ and $S\cap B\neq\emptyset$, then one of the following conditions holds:
		\begin{enumerate}
			\item Each vertex of $K$ is black; or
			\item $K$ contains $K_{1,t}$ as a spanning tree for some positive integer $t$.
		\end{enumerate}	
	\end{lemma}
	\begin{proof}
		If condition 1. holds, then there is nothing to prove. Suppose that there exists a non-black vertex $v$ in $V(K)$. Then $v$ is yellow, as $R\cap B\neq\emptyset$ and $S\cap B\neq\emptyset$. Suppose, without loss of generality, that $v\in S$. If $|R|\ge 2$ and $|S|\ge 2$, then $v$ has at least two non-white neighbors in $R$, leading to a contradiction. The contradiction arises from assuming that $|R|\ge 2$ and $|S|\ge 2$. Hence $|R|=1$ or $|S|=1$, and consequently, $K$ contains $K_{1,t}$ as a spanning tree for some positive integer $t$.
	\end{proof}
	If each vertex of $K$ is black, then the unique PED-set is the trivial one. Now, we assume that $K$ is a maximal $K_{1,t}$ dominating complete bipartite subgraph. 
	
	By $G_0$, we denote the subgraph $G-V(K)$. Let $C$ be a connected component of $G_0$. If $C$ is bipartite, we denote the bipartition of $C$ with $\{U_C,V_C\}$. Let us assume, without loss of generality, that $|R|=1$, with $R=\{r\}$. 
	
	We use $S_1$ to denote the set of vertices $v$ of $S$ such that $N^S(v)\neq\emptyset$. Set $S_2=S\setminus S_1$. By connected component of $G_0$ of \emph{type $\alpha$}, we mean those non-trivial connected components $C$ of $G_0$ such that there exists a vertex $v\in S$ with at least one neighbor and at least one non-neighbor in $V(C)$. A non-trivial connected component that is not of type $\alpha$ is called of \emph{type $\beta$}.
	\begin{lemma}\label{lem: several black}
		Let $G$ be a connected $P_6$-free graph with a maximal dominating complete bipartite subgraph $K$, having a bipartition $\{R=\{r\},S=N(r)\}$, and let $t=|S|$. Suppose that $G$ admits a valid $3$-coloring $\pi$ such that $\pi(r)=b$ and $S\cap B\neq\emptyset$. Then the following conditions hold:
		\begin{enumerate}
			\item If $t=1$, then $W=\emptyset$. \label{lem14: item 1}
			\item $S_1\subseteq B$. \label{lem14: item 2}
			\item\label{lem14: item 3} Assume $t\ge 2$ and that every connected component of $G_0$ is trivial. Consider the induced subgraph $G-\{r\}$, and let $Z=V(G)\setminus V(K)$, $Z'=\{z\in Z: d_G(z)=1\}$, and $Z''=Z\setminus Z'$. Then: 
			\begin{itemize}
				\item Every connected component $C$ of $G-\{r\}$ that contains a vertex from $S_1$ contains no white vertex; i.e., $V(C)\cap (S\cup Z'')\subseteq B$ and $V(C)\cap Z'\subseteq Y$.
				\item Every trivial connected component $C$ of $G-\{r\}$ consists of a single vertex $s\in S_2$, and in that case, $s\in Y$.
				\item Any other connected component $C$ of $G-\{r\}$ can be colored independently of the others, choosing one of the following options:
				\begin{enumerate}
					\item $V(C)\cap (S\cup Z'')\subseteq B$ and $V(C)\cap Z'\subseteq Y$.
					\item $V(C)\cap S\subseteq Y$ and $V(C)\cap Z\subseteq W$
				\end{enumerate}
			\end{itemize}
			Additionally, if no connected component chooses option (b), then $W=\emptyset$. Moreover, if $S_1=\emptyset$ and every non-trivial component chooses option (b), then the resulting coloring is valid and uses fewer edges than any other combinations. However, $S\cap B=\emptyset$ and the hypothesis of the lemma is not satisfied.
			\item ~\label{lem14: item 4}
			\begin{enumerate}			
				\item If some connected component $C$ of $G_0$ has a black vertex, then all vertices in $C$ are black.
				\item $V(C)\subseteq B$ for every non-bipartite connected component $C$, and for every connected component $C$ of type $\beta$.
				\item Let $C$ be a connected component of $G_0$ of type $\alpha$, with two adjacent vertices $c$ and $d$, and a vertex $v\in S$ such that $v$ is adjacent to $c$ and non-adjacent to $d$. If $u\in S_2$ is a vertex with at least one neighbor in $V(G_0)\setminus V(C)$, then either $N^{V(G_0)}(v)\cap N^{V(G_0)}(u)\neq\emptyset$, or $u$ is adjacent to $d$. Furthermore, if $|N^S(c)|\ge 2$ and $|N^S(d)|\ge 2$ or if $C$ is non-bipartite, then $V(C)\subseteq B$.
				\item If $C$ is a connected component of $G_0$ of type $\alpha$ and $V(C)\subseteq B$ then $\pi(v)\neq w$, for all $v\in V(G)$. 
			\end{enumerate}
		\end{enumerate}
	\end{lemma}
	\begin{proof}
		If $t=1$ and $S\cap B\neq\emptyset$, then $S\cap B=\{s\}$. Hence, each vertex in $V(G)\setminus V(K)$, that is a neighbor of $s$, has a non-white color, leading to $W=\emptyset$. This proves Statement \ref{lem14: item 1}.
		
		Statement~\ref{lem14: item 2} trivially holds.
		
		Note that each vertex in $Z'$ is adjacent to only one vertex in $S$, so $Z\subseteq Y\cup W$
		 Let $C$ be a connected component of $G-\{r\}$ containing a vertex $v\in S_1$. By Statement~\ref{lem14: item 2}, $v$ is black, and each of its neighbors in $V(C)\setminus S$ must be non-white. Furthermore, any vertex $u$ in $Z''$ with a black neighbor is also black, as is any vertex $u\in S_2$. These observations imply that if $C$ contains a vertex from $S_1$, then $V(C)\cap (S\cup Z'')\subseteq B$ and $V(C)\cap Z'\subseteq Y$. It is also clear that any trivial connected component of $G-r$ consists of a single vertex from $S_2$.
		 
		 A similar argument applies to the other components. If at least one vertex in $V(C)\cap (Z''\cup S)$ is black, then all vertices in $V(C)\cap (Z''\cup S)$ must also be black, and all vertices in $V(C)\cap Z'$ are yellow. Otherwise, all vertices in $V(C)\cap S$ are yellow, and all vertices in $V(C)\cap Z$ are white. 
		
		Statement~4 can be proved using arguments analogous to those used in Lemmas~\ref{lem: K white and yellow} and~\ref{lem: components of type I}.
	\end{proof}
	
	\section{Unweighted minimum PED}\label{sec: the algorithm}
	In this section, we present a cubic-time algorithm for solving the unweighted minimum PED-problem on $P_6$-free graphs. The algorithm relies on Theorem~\ref{p6-free struct}. Moreover, if a dominating $C_6$ or a dominating complete bipartite is provided, the algorithm runs in linear-time.
	
	If a graph $G$ has a DIM $D\subseteq E(G)$, then $D$ is a minimum PED-set (see,~\cite{Lu-Ko-Tang-2002}). 
%

	We assume, without loss of generality, that the input graph $G=(V(G),E(G))$ is a connected $P_6$-free graph, where $n=|V(G)|$, and $m=|E(G)|$.
	
	\vspace{0.25cm}
	{\bfseries Main algorithm}
	\begin{enumerate}
		\item Run the $O(n+m)$-time algorithm from~\cite{Br-Mo-2014}, noting that every $P_6$-free graph is a $P_7$-free graph. If the algorithm finds a DIM $D\subseteq E(G)$ of $G$, then $D$ is the optimal solution. The subsequent steps assume that $G$ is a DIM-less graph and proceed to iteratively construct iteratively PED-sets, keeping the track of the best one found so far. The best current PED-set of $G$ is denoted by $\mathcal{P}$. 		
		\item Set $\mathcal{P}:=\{E(G)\}$ (the trivial PED-set). 
		This operation requires $O(m)$-time.
		\item Apply the algorithm of Theorem \ref{p6-free struct} 
		\cite{DBLP:conf/cocoon/HofP08,DBLP:journals/dam/HofP10} to $G$. One of the following structures will be returned in $O(n^3)$ time:
		\begin{itemize}
			\item A dominating induced cycle $C=(v_1,v_2,v_3,v_4,v_5,v_6)$ with six vertices. Proceed to step 4.
			\item A dominating complete bipartite subgraph $K$ with a bipartition $\{R,S\}$. We can assume that $K$ is maximal with additional $O(n+m)$-cost and proceed to step 6.
		\end{itemize}
		\item For each possible partial $3$-coloring of $C$ of types (b), (c) and (e), attempt to extend it to a total $3$-coloring of $G$ that corresponds to a proper PED-set according to Theorem~\ref{thm: dominating cycle}. If the extension is valid, compare it with the best current PED-set $\mathcal{P}$ and update $\mathcal{P}$ if the new one is better. Validity checking of a total $3$-coloring requires $O(m+n)$-time. 
		
		The number of possible total $3$-colorings is $O(1)$ except those of type (e). Since the number of possible colorings for each type (b) and (c) is $6$ (see Lemma~\ref{lem: coloreos del C_6} and Theorem~\ref{thm: dominating cycle}), each extension can be done in $O(m+n)$-time. Therefore, the generation and validation of these 12 total $3$-colorings can clearly be done in overall $O(m+n)$-time.

		Next, we analyze partial $3$-colorings of type (e) for $C$. There are exactly two such colorings. For each of them, one of the following situations occurs (see Theorem \ref{thm: dominating cycle}):
			\begin{enumerate}
				\item No valid total $3$-coloring can be extended.
				\item Exactly one valid total $3$-coloring can be extended, which corresponds to a DIM (i.e., $G[Y]$ is $1$-regular).
				\item Multiple total $3$-colorings can be extended ($G[Y]$ is not $1$-regular). If one of them is valid, then all of them are valid and correspond to a proper PED-sets. These PED-sets have the same number of edges, since the corresponding valid $3$-colorings share the same set of yellow vertices and exactly one black vertex with a specific degree. In this case, it suffices to consider any one of them.
			\end{enumerate}
Again, at most two significant total $3$-colorings (corresponding to type (e)) are generated and checked in $O(n + m)$ time.
		\item Return $\mathcal{P}$ and stop.
		\item Consider the following three configurations for $K$:
		\begin{enumerate}
			\item Configuration I: $K$ has a partial $3$-coloring with no black vertices. Perform the non-black-vertex sub-algorithm. See Subsection 4.2.1.
			\item Configuration II: $K$ has exactly one black vertex. Perform the one-black-vertex sub-algorithm. See Subsection 4.2.2.
			\item Configuration III: $K$ has at least two black vertices and is isomorphic to $K_{1,t}$. Perform the several-black-vertices sub-algorithm. See Subsection 4.2.3.
		\end{enumerate}
		\item Return $\mathcal{P}$ and stop.
	\end{enumerate}
	\vspace{0.25cm}
	{\bfseries Non-black-vertex sub-algorithm (Configuration I)}
	\vspace{0.25cm}
	
	The list $\Pi$ contains valid partial $3$-colorings found so far. At various stages, initial valid partial $3$-colorings are generated and stored in $\Pi$. Subsequent steps attempt to extend these partial $3$-colorings into valid total $3$-colorings of $G$, or determines that such an extension is not possible. 
	
	Note that, under the hypothesis of Lemma~\ref{lem: valid 3-coloring}, if the graph admits a non-trivial solution, then it also contains a DIM, which cannot occur, since after the first step of the main algorithm we assume that the input graph does not admit any DIM. Therefore, the second condition of Lemma~\ref{lem: dominating star} must hold, that is, $R\subseteq W$ or $S\subseteq W$.
	\begin{enumerate}	
		\item Set $\Pi:=\emptyset$.
		%
		%
		\item Assign white color to the vertices of $R$, and a yellow color to the vertices of $S$. Then, invert the coloring of each set in the bipartition (see Lemma \ref{lem: K white and yellow}), and repeat the following substeps for each of these two initial valid partial $3$-colorings:
		\begin{enumerate}	
			\item Let $X$ be the subset of vertices in $V(G)\setminus V(K)$ with at least one neighbor in $R$, and assign yellow color to vertices of $X$. 
			\item Check that $X\cup S$ is a dissociation set: compute $G[X\cup S]$, $S_1 := \{s \in S \mid d_{G[X\cup S]}(s)=1\}$, $S_2 := \{s \in S \mid d_{G[X\cup S]}(s)=0\}$, $X_1 := \{x \in X \mid d_{G[X\cup S]}(x)=1\}$ and $X_2 := \{x \in X \mid d_{G[X\cup S]}(x)=0\}$. If it fails (that is, $S\neq S_1\cup S_2$ or $X\neq X_1\cup X_2$), then skip to the next initial valid partial $3$-coloring of the step (this one cannot be extended to a valid total $3$-coloring of $G$).
			\item Compute the set $\mathcal{C}$ of connected components of $G_0=G-(V(K) \cup X)$.
			\item If a vertex $v\in S_2\cup X_2$ has $d^{V(G_0)}_G(v)=0$, then skip to the next initial valid partial $3$-coloring of the step (this one cannot be extended to a valid total $3$-coloring of $G$). See Remark \ref{rmk: invalid coloring}.			
			\item Assign black color to every vertex of each non-bipartite connected component $C\in\mathcal{C}$ (see Lemma~\ref{lem: K white and yellow}).
			\item For each $v\in S_1\cup X_1$, assign white color to each vertex $u\in N_G(v)\cap V(G_0)$, and assign a white/yellow coloring to  the bipartite connected component $C(u)$ of $u$ in $G_0$, compatible with $u$, whenever $u$ is uncolored. If some $u\in N(v)\cap V(G_0)$ was previously assigned a non-white color, skip to the next initial valid partial $3$-coloring of the step (this one cannot be extended to a valid total $3$-coloring of $G$).
			\item If all connected components of $\mathcal{C}$ are now colored, store the partial $3$-coloring in $\Pi$. Otherwise, let $C \in \mathcal{C}$ be an uncolored component of type I. In this case, there exists some vertex $v \in S_2$ that has a neighbor $c$ and a non-neighbor $d$ in $C$, such that $cd$ is an edge of $C$. If $d$ has a exactly one neighbor $v'$ in $S_2$, and $v'$ is adjacent to $c$, then assign yellow to each vertex in the partition class of $C$ containing $d$, and white to each vertex in the other class ($c$ has 2 yellow neighbors $v$ and $v'$ which implies that $c\notin Y$, as $\{v',c,d\}$ forms a triangle which implies that $c\notin B$, $c\in W$ and $d\in Y$). Repeat this procedure as long as components satisfying this condition remain. Otherwise, check the resulting partial $3$-coloring is valid (that is any yellow vertex has at most one colored non-white neighbor and the colored non-yellow vertices form an independent set)? If the check fails, then skip to the next initial valid partial $3$-coloring of the step (this one cannot be extended to a valid total $3$-coloring of $G$). Otherwise, consider up to three extensions of the current partial $3$-coloring as follows:
			\begin{enumerate}
				\item Assign black to all vertices of $C$. If there is some vertex $z\in S_2\cup X_2$ with $d^{V(C)}_G(z)\geq 2$ then skip to the next alternative (that partial $3$-coloring cannot be extended to a valid $3$-coloring of $G$). Otherwise, for each $z\in S_2\cup X_2$ with $d^{V(C)}_G(z)=1$, assign white color to each $u \in N(z)\cap (V(G_0)\setminus V(C))$, and assign a compatible white/yellow coloring to $C(u)$ compatible with $u$, whenever $u$ is uncolored. If any such $u$ was previously colored non-white (that partial $3$-coloring cannot be extended to a valid $3$-coloring of $G$) skip to the next alternative. Otherwise, store it in $\Pi$.
				
				
				%
				\item If $c$ has exactly one neighbor in $S_2$, do the following: assign yellow to $c$ and assign a compatible white/yellow coloring to $C(c)$ compatible with $c$. Then assign white to each vertex $y\in N(v)\cap (V(G_0)\setminus\{c\})$, and assign a white/yellow coloring to $C(y)$ whenever $y$ is uncolored. If any such $y$ was previously colored with a non-white color (that partial $3$-coloring cannot be extended to a valid $3$-coloring of $G$) skip to the next alternative. Otherwise, store the resulting partial $3$-coloring in $\Pi$. 
				\item If $d$ has exactly one neighbor $v'\in S_2$ (clearly, $c\notin N_G(v')$ and $v'$ is another indicator of $C$), do the following: assign yellow to $d$ and a compatible white/yellow coloring to $C(d)$ compatible with $d$. Then assign white to each vertex $y\in N(v')\cap (V(G_0)\setminus\{d\})$, and assign a compatible white/yellow coloring to $C(y)$ whenever $y$ is uncolored. If any such $y$ was previously assigned a non-white color (that partial $3$-coloring cannot be extended to a valid $3$-coloring of $G$) skip to the next initial valid partial $3$-coloring. Otherwise, store that partial $3$-coloring in $\Pi$.
				\end{enumerate}
			\end{enumerate}
		The time complexity of this step is $O(m+n)$. The list $\Pi$ has at most 6 partial $3$-colorings.

		See Lemma~\ref{lem: components of type I} and Remark~\ref{rmk: components of type I} for the correctness. Note also that, if $d$ has more than one neighbor in $S_2$ and all of them are adjacent to $d$, then it leads to an invalid $3$-coloring. Otherwise, by Remark~\ref{rmk: components of type I} there is no uncolored components of type I.
		\item If $\Pi = \emptyset$, this indicates that there is no non-trivial PED-set extension of $K$, and the algorithm should return to the main procedure.
		\item For each valid partial $3$-coloring of $\Pi$, repeat the following substeps.
		\begin{enumerate}
			\item For each non-trivial uncolored component $C$ (which must be of type II), note by Lemma~\ref{lem: type II} that $|U_C|=1$ or $|V_C|=1$. Select the first of the following statements that is true and take the corresponding actions:			
			\begin{enumerate}
				\item If $U_C=\{u\}$ and $d^S(u)=1$: we assign yellow to $u$ and white to each vertex in $V_C$.
				\item If $V_C=\{v\}$ and $d^S(v)=1$: we assign yellow to $v$ and white to each vertex in $U_C$.			
				\item Otherwise, discard that partial $3$-coloring.
			\end{enumerate}
			\item Compute $S^\pi_2$, $X^\pi_2$ and $Z$, where $\pi$ is the resulting partial $3$-coloring of (a). Consider the graph $H_\pi$. For each connected component $C$ of $H_\pi$ that contains some vertex in $S^\pi_2\cup X^\pi_2$, find a vertex $v$ that is adjacent to every vertex in $V(C)\cap (S^\pi_2\cup X^\pi_2)$. Assign the yellow color to $v$ if $d_C(v) = 1$, or the black color if $d_C(v) > 1$, and assign the white color to every other uncolored vertex in that component. If no such a vertex exists for some connected component, discard that partial $3$-coloring (see Lemma \ref{lem: optimum coloring}).
			\item If the resulting $3$-coloring is a valid total $3$-coloring, compare it with the current best PED-set $\mathcal{P}$ and update $\mathcal{P}$ if it is better.
		\end{enumerate}
	\end{enumerate}
	This sub-algorithm can be implemented in $O(m+n)$-time.

	\vspace{0.25cm}
	{\bfseries One-black-vertex sub-algorithm (Configuration II)}
	\vspace{0.25cm}
	
	Refer to Lemma \ref{lem: one black vertex} for the correctness of the following subroutine. Consider the bipartition $\{R,S\}$ of $K$ and apply the following procedure. Then, repeat it interchanging the roles between $R$ and $S$.
	\begin{enumerate}	
		\item Check that $R$ and $S$ are both independent sets. If it fails, then it is not possible to extend to a valid $3$-coloring.
		\item $M=\{v\in V(G)\setminus V(K): N^S_G(v)\ne\emptyset\}$ and check that $R\cup M$ is an independent set. Again, if it fails, then it is not possible to extend to a valid $3$-coloring. 
		\item $S\subseteq Y$ and $M\subseteq W$.
		\item $N=V(G)\setminus (R\cup S\cup M)$ and consider the degree of each vertex $x$ in the induced subgraph $G[N]$, i.e., $d^N(x)=|N_{G[N]}(x)|$.
		\begin{enumerate}
			\item If $d^N(x)\geq 2$, then $x\in B$, $x$ has exactly one neighbor in $K$ and this neighbor is $r\in R$ (the unique black vertex of $R$, and $R\setminus \{r\}\subseteq W$). 
			\item If $d^N(x)=0$, then $x\in Y$ and $x$ has $r$ as neighbor.
			\item If $d^N(x)=1$ and $x$ has at least two neighbors in $R\cup M$, then $x\in Y$ and $N^R_G(x)\subseteq W$ ($x$ is non-adjacent to $r$).
		\end{enumerate}
		Any vertex $r'$ of $R$ that have vertices of (a) and (b) as neighbors and vertices of (c) as non-neighbors can be selected as $r$ and complete a valid $3$-coloring as follow: $r\in B$, $R\setminus \{r\}\subseteq W$, and any neighbor (non-neighbor) $x$ of $r$ with $d^N(x)=1$ must be black (yellow). It is always advisable to choose $r'$ with lowest degree in $G$ as $r$ to obtain a smaller PED-set.
	\end{enumerate}
	This sub-algorithm runs in $O(m+n)$.
	
	\vspace{0.25cm}
	{\bfseries Several-black-vertices sub-algorithm (Configuration III)}
	\vspace{0.25cm}

	\begin{enumerate}	
		\item By Lemmas~\ref{lem: star - black vertices} and \ref{lem: several black}, we assume that the bipartition of $K$ is $\{R=\{r\},S=N_G(r)\}$ and $t=d_G(r)\geq 2$.
		\item Compute $S_1$, $S_2$, and all connected components of $G_0 = G - K$. If $S_1=\emptyset$ and $G_0$ contains only trivial connected components, then stop. In this situation, there is a better PED-set but it comes from a different configuration (Lemma~\ref{lem: several black}, item~\ref{lem14: item 3}).
		\item Set $X:=\emptyset$ to represent the subset of vertices that cannot be colored white.
		\item Assign the black color to $r$ and to all vertices in $S_1$.
		\item Assign the yellow color to the subset of vertices of $S_2$ such that do not have any neighbor in $G_0$.
		\item Let $Z$ be the set of isolated vertices of $G_0$.
		\item Add to $X$ every vertex $z\in Z$ such that $N(z)\cap S_1\neq\emptyset$.
		\item If all non-trivial connected components of $G_0$ are of type $\beta$, proceed as follows:
		\begin{enumerate}
			\item Assign the black color to every vertex in each non-trivial connected component.
			\item Assign the black color to every vertex $v\in S_2$ that has a neighbor in $X$ or in $B\setminus \{r\}$. For each such $v$, add all its uncolored neighbors to $X$. Apply this procedure iteratively until no more vertices of $S_2$ can be assigned as black vertices.
			\item Let $H$ be the subgraph induced by $S_2'$ (the uncolored vertices of $S_2$) and $Z\setminus X$. Compute its connected components. Any connected component $C$ can take two possible options for its vertices independently of other connected components:
			\begin{enumerate}
			\item $V(C)\cap S_2'\subseteq B$ and add $V(C)\setminus S_2'$ to $X$.
			\item $V(C)\cap S_2'\subseteq Y$ and $V(C)\setminus S_2'\subseteq W$.
			\end{enumerate}
			Any of these combinations can complete a valid $3$-coloring: assign either yellow or black to each vertex in $X$, depending on whether it has exactly degree one in $G-W$ or not. The corresponding PED-set is non-trivial if and only if some connected component $C$ has taken (ii) as option. For unweighted version problem, it is always convenient to choose option (ii) for every connected component. Consequently, $S_2'\subseteq Y$ and $Z\setminus X\subseteq W$.
			%
			%
			%
			\item Compare this PED-set with the current best PED-set $\mathcal{P}$ and update $\mathcal{P}$ if it is better.
		\end{enumerate}
		\item  If $G_0$ contains a connected component $C$ of type $\alpha$, find a vertex $v\in S$ with two adjacent vertices $c,\, d\in V(C)$ such that $v$ is adjacent to $c$ and non-adjacent to $d$. Then, proceed as follows:
		\begin{enumerate}
			\item If $C$ is non-bipartite, or if both $c$ and $d$ have at least two neighbors in $S$, then stop. In this case, the only valid total $3$-coloring corresponds to the trivial PED-set (Lemma~\ref{lem: several black}, item~\ref{lem14: item 4}).
			\item Otherwise, $C$ is bipartite and explore at most two possible extensions of the partial $3$-coloring:
			\begin{enumerate}
				\item If $c$ has only one neighbor $v$ in $S$, then:
				\begin{itemize}
					\item  Assign yellow to $c$ and extend to a compatible yellow/white assignment coloring for $C(c)$.
					\item Assign black to $v$ and add to $X$ each uncolored neighbor of $v$ in $G_0$.
					\item For each neighbor $u$ of $d$ in $S$ (there is at least one such vertex), assign yellow to $u$ and white to each neighbor of $u$ except $r$.
					\item For any uncolored vertex $s\in S_2$, one of the following statements is true if it is possible to extend the partial coloring to a valid coloring:
					\begin{itemize}
						\item Its neighbors in $G_0$ are vertices of $C$ and they are yellow vertices. This implies that $s\in B$.
						\item Its neighbors in $G_0$ are vertices of $C$ and they are white vertices. This implies that $s\in Y$.
						\item It shares a white neighbor with some $u$ and it implies that $s\in Y$ and its uncolored neighbors are white vertices.
					\end{itemize}
					\item If any inconsistency arises, discard the partial coloring. Otherwise, assign consistent yellow/black colors to the vertices in $X$.
					\item If the resulting $3$-coloring is valid, compare it with the current best PED-set $\mathcal{P}$ and update if necessary.
				\end{itemize}
				\item If $d$ has only one neighbor $v'$ in $S$, proceed analogously:
				\begin{itemize}
					\item Assign yellow to $d$ and extend to a compatible yellow/white assignment for $C(d)$. 
					\item Assign black to $v'$ and add to $X$ each uncolored neighbor of $v'$ in $G_0$.
					\item For each neighbor $u$ of $c$ in $S$ (in particular $v$ is a neighbor of $c$), assign yellow to $u$ and white to each neighbor of $u$ except $r$.
					%
					\item For any uncolored vertex $s\in S_2$, one of the following statements is true if it is possible to extend the partial coloring to a valid coloring:
					\begin{itemize}
						\item Its neighbors in $G_0$ are vertices of $C$ and they are yellow vertices. This implies that $s\in B$.
						\item Its neighbors in $G_0$ are vertices of $C$ and they are white vertices. This implies that $s\in Y$.
						\item It shares a white neighbor with $u$ and it implies that $s\in Y$ and its uncolored neighbors are white vertices.
					\end{itemize}
					\item If any inconsistency arises, discard the partial coloring. Otherwise, assign consistent yellow/black assignment to the vertices in $X$.
					\item If the resulting $3$-coloring is a valid, compare it with the current best PED-set $\mathcal{P}$ and update if necessary. 
				\end{itemize}
			\end{enumerate}
		\end{enumerate}
	\end{enumerate}
	This sub-algorithm can be implemented in $O(m+n)$-time.	
	\vspace{0.25cm}
	
	As a consequence of the correctness of these algorithm, Theorem \ref{thm: main result} has already been proven.
	
\section{Weighted minimum PED-set and Counting number of PED-sets}\label{sec: weighted and counting version}
	
In this section, we describe the adjustments and adaptations needed for the polynomial-time algorithm presented in the previous section so that it can also solve the weighted minimum PED-set problem and count the number of PED-sets for $P_6$-free graphs, without increasing the time complexity.

We begin by listing the parts of the algorithm that need to be modified, together with briefly explanation of how to implement these changes while preserving the algorithm's efficiency. Three of these modifications require a more detailed discussion, and each of them is addressed in a separate subsection.

In essence, the algorithm starts from a dominant subgraph and considers different possible partial colorings, attempting to extend them deterministically. However, branching is sometimes unavoidable, as several alternative extensions may be possible. Most of these branchings have a constant branching factor. The current best PED-set $\mathcal{P}$ now stores the minimum weight PED-set found so far. In addition, a new variable, $\#\mathcal{P}$, is introduced to record the number of PED-sets found up to the current stage.

\begin{enumerate}
	\item Step~1 of the main algorithm uses $O(n+m)$-time algorithm from~\cite{Br-Mo-2014} to find a DIM of the input graph $G$. In the positive case, the unweighted problem is solved. For the weighted and counting problems, however, this approach is no longer appropriate, since DIMs can have different weights and a lighter DIM is not necessarily a lighter PED-set. More configurations must therefore be taken into account. Two parts of the algorithm require modifications:
		\begin{enumerate}
			\item In Step~4 of the main algorithm, three partial colorings of $C$ of type~(d) must be considered. Each can be extended to at most one valid coloring in a similar way to configurations~(b) and~(c). The extra cost is $O(n+m)$. 
			\item At the beginning of non-black-vertex sub-algorithm (Configuration~I) the hypothesis of Lemma~\ref{lem: valid 3-coloring} may now hold, so extendible colorings under these conditions should also be considered. See details in Subsection~\ref{subsec: details One}. 
		\end{enumerate}
Instead of the execution of the $O(n+m)$-time algorithm from~\cite{Br-Mo-2014} to find a DIM, we define a vertex-weight function $\psi:V(G)\longrightarrow\mathbb{R}$ using the edge-weighted function $\omega:E(G)\longrightarrow\mathbb{R}$ as follow: $\psi(u)=\sum_{v\in N_G(u)}\omega(uv)$, for $u\in V(G)$. These weights can be computed in $O(n+m)$-time.
%
\item In Substep~4.a of non-black-vertex sub-algorithm (Configuration~I), statements~(i) and~(ii) should each be evaluated independently, taking their corresponding actions if true (each could lead to a different valid $3$-coloring). This adjustment does not affect the algorithm's complexity.
\item In Substep~4.b of the non-black-vertex sub-algorithm (Configuration I), for each connected component $C$ of $H_\pi$ only one PED-set of $C$ is currently considered. All PED-sets of $C$ must now be taken into account. See details in Subsection~\ref{subsec: details Two}.
\item In Step~4 of one-black-vertex sub-algorithm (Configuration~II), instead of choosing $r'$ with the lowest degree as $r$, choose the one with smallest weight $\psi(r')$, and add the number of possible $r'$ to $\#\mathcal{P}$. These changes do not alter the time complexity. 
\item In Step~2 of the several-black-vertices sub-algorithm (Configuration~III), ignore the stopping condition when $S_1=\emptyset$ and $G_0$ contains only trivial connected components.
\item In Substep 8.c of the several-black-vertices sub-algorithm (Configuration~III), it is no longer always optimal to choose option (ii) for each connected component $C$ of $H$.  For each $C$, compute the weight difference between option~(i) and option~(ii), which is exactly $\psi(V(C)\setminus S_2')$. Option~(i) is preferable for $C$ if and only if this value is negative. Add $2^{\#C}$ to $\#\mathcal{P}$, where $\#C$ is the number of connected components of $H$. Subtract one from $\#\mathcal{P}$ if $S_1=\emptyset$ and $G_0$ contains only trivial connected components. The extra cost is $O(m+n)$.
\end{enumerate}

\subsection{Detailed explanations of Modification (1.b)}\label{subsec: details One}

The hypothesis of Lemma~\ref{lem: valid 3-coloring} is as follows: $G$ is a $P_6$-free graph containing a maximal dominating (not necessarily induced) complete bipartite subgraph $K$ (with partitions $R$ and $S$), and $G$ admits a valid $3$-coloring $(B,Y,W)$ such that $|R\cap Y| = |S\cap Y| = 1$. This lemma guarantees that: 
\begin{enumerate}
	\item $|R| = 1$ or $|S| = 1$, and
	\item the corresponding PED-set is a DIM.
\end{enumerate}
We consider the following cases:
\begin{itemize}
\item \textbf{Case 1:} $|R|=|S|=1$. Clearly, $V(G)\setminus V(K)\subseteq W$ and $B=\emptyset$. If this unique $3$-coloring is valid, then its corresponding PED-set is a DIM consisting of exactly one edge, namely the unique edge of $K$.
\item \textbf{Case 2:} $|R|>1$ and $|S|=1$. Swap the labels of $R$ to $S$ and consider Case~3.
\item \textbf{Case 3:} $|R|=1$ and $|S|>1$. Since $K$ is maximal, we have $R=\{r\}$ and $S=N_G(r)$. Clearly, $r\in Y$. By Lemma~\ref{lem: valid 3-coloring}, we analyze all possible valid $3$-colorings according to the number of edges in $G[S]$. 

If $G[S]$ contains at least one edge, then the lemma describes a method to find at most two different possible valid $3$-colorings, and this can be done in $O(m+n)$-time. 

Now suppose that $G[S]$ has no edges. In this case, the lemma provides a procedure to find 
\[
W', W'' \subseteq S \cap W
\quad\text{and}\quad
Y', Y'' \subseteq (V(G) \setminus V(K)) \cap Y.
\]
This procedure can also be executed in $O(m+n)$ time.

If $Y'\cup Y''$ is not a dissociation set, or $W'\cup W''$ is not an independent set or $S'=S\setminus (W'\cup W'')=\emptyset$, then there is no valid $3$-coloring that satisfies the hypothesis of Lemma~\ref{lem: valid 3-coloring}. This check can be performed in $O(n+m)$-time. 

If instead $Y' \cup Y''$ is a dissociation set, $W' \cup W''$ is an independent set, and $S' \neq \emptyset$, then the task reduces to finding the vertices $s \in S'$ such that the corresponding $3$-coloring
\[
B = \emptyset, \quad 
Y = ((V(G) \setminus S) \setminus N_G(s)) \cup \{r,s\}, \quad
W = (S \cup N_G(s)) \setminus \{r,s\}
\]
is valid, and computing the weight of its corresponding DIM. The number of such vertices is added to $\#\mathcal{P}$, and $\mathcal{P}$ is updated if necessary.

Below is an $O(m+n)$-time procedure, \emph{One}, that determines all such vertices and the weights of their corresponding DIMs. The correctness and linear-time complexity of One follow from the facts that, for any two distinct vertices in $S'$, $r$ is their only common neighbor, and the neighborhood of each vertex in $S'$ is an independent set.
\end{itemize} 
%
		\begin{algorithmic}[1]
		\Procedure{One}{}
		\State $L \gets \emptyset$
		\State $partial\_sum \gets \frac{\psi(V(G))}{2}-\psi(S)$
		\State $\#pendant \gets 0$
		\For{$v \in V(G)\setminus V(K)$} 
		\State $deg[v] \gets |N_{G-V(K)}(v)|$
		\If{$deg[v] = 1 $}
		\State $\#pendant \gets \#pendant + 1$
		\EndIf
		\EndFor
		\For{$s\in S'$}
		\For{$w\in N_G(s)\setminus \{r\}$}
		\For{$z\in N_{G-V(K)}(w)$}
		\State $deg[z] \gets deg[z] - 1$
		\If{$deg[z] \leq 1 $}
		\State $\#pendant \gets \#pendant + 2 * deg[z] - 1$
		\EndIf
		\State $deg[w] \gets deg[w] - 1$
		\If{$deg[w] \leq 1 $}
		\State $\#pendant \gets \#pendant + 2 * deg[w] - 1$
		\EndIf	
		\EndFor
		\EndFor
		\If{$\#pendant = n-|S|-|N_G(s)| $} \Comment the yellow vertices define a DIM, compute the weight of such DIM and keep the pair $(s,weight)$ in $L$ 
		\State $L \gets L \cup \{(s,partial\_sum+\psi(s)-\psi(N_G(s)\setminus \{r\}))\}$
		\EndIf
		\For{$w\in N_G(s)\setminus \{r\}$}
		\For{$z\in N_{G-V(K)}(w)$}
		\If{$deg[z] \leq 1 $}
		\State $\#pendant \gets \#pendant - 2 * deg[z] + 1$
		\EndIf
		\State $deg[z] \gets deg[z] + 1$
		\If{$deg[w] \leq 1 $}
		\State $\#pendant \gets \#pendant - 2 * deg[w] + 1$
		\EndIf
		\State $deg[w] \gets deg[w] + 1$
		\EndFor
		\EndFor	
		\EndFor
		\State 
		\Return $L$
		\EndProcedure
	\end{algorithmic}

\subsection{Detailed explanations of Modification (3)}\label{subsec: details Two}

In this subsection, we describe how to efficiently determine: (i) the number of distinct PED-sets of $H_\pi$ that correspond to valid $3$-colorings with $S_2\cup X_2\subseteq Y$, and (ii) one such PED-set with minimum weight. 

Clearly, the value in (i) is obtained by multiplying the numbers of valid $3$-colorings in each connected component $C$ of $H_\pi$ where $V(C)\cap (S_2\cup X_2)\subseteq Y$. 

We use the following recursive procedure, \emph{Two} (with $S^\pi_2, X^\pi_2$, and $Z$ as arguments), to compute this number and find the PED-set for (ii). This procedure returns a number $Qty$, a subset $Z'\subseteq Z$, and the weight of the optimal DIMs. Here, $Qty$ is the answer to (i), and when $Qty\geq 1$, $Z'$ completes a valid $3$-coloring of $H_\pi$ by assigning: 
\begin{itemize}
	\item white color to $Z\setminus Z'$,
	\item  yellow color to $\{z\in Z' : d_G(z)=1\}$, and
	\item black to $\{ z \in Z' : d_G(z) \geq 2 \}$.
\end{itemize}

The  PED-set corresponding to this $3$-coloring is the answer to (ii). The correctness of Two is due to Lemma~\ref{lem: optimum coloring}, and its time complexity is $O(m*n)$, since there are at most $|Z'|+1$ calls to the procedure.
		\begin{algorithmic}[1]
		\Procedure{Two}{$S,\,X,\, Z$}
		\If{$Z = \emptyset$}
		\Return $(0,\emptyset, 0)$
		\EndIf
		\State Compute the induced subgraph $H=G[S\cup X \cup Y]$ and $\mathcal{C}$ its connected components
		\State $Qty \gets 1$
		\State $Z' \gets \emptyset$
		\State $Weight \gets 0$
				\For{$C \in \mathcal{C}$} 
		\If{$Z\cap V(C) = \emptyset$}
		\Return $(0,\emptyset, 0)$
		\EndIf
		\State $z \gets$ a vertex of $Z\cap V(C)$ with maximum degree in $C$
		\If{$d_C(z) < |(S\cup X)\cap V(C)|$}
		\Return $(0,\emptyset, 0)$
		\EndIf
		\State $(auxQty,auxZ,auxWeight)  \gets Two(S \cap V(C), X \cap V(C), (Z\cap V(C))\setminus\{z\})$
		\If{$auxQty = 0$}
		\State $Z' \gets Z'\cup \{z\}$
		\State $Weight \gets Weight+\psi(z)$
		\Else
		\State $Qty \gets Qty*(1+auxQty)$
		\If{$\psi(z) \leq auxWeight$}
		\State $Z' \gets Z'\cup \{z\}$
		\State $Weight \gets Weight+\psi(z)$
		\Else
		\State $Z' \gets Z'\cup auxZ$
		\State $Weight \gets Weight+auxWeight$
		\EndIf
		\EndIf
		\EndFor
		\State
		\Return $(Qty,Z', Weight)$
		\EndProcedure
	\end{algorithmic}
 
\section{Conclusions and further results}\label{sec: conclusions}
To the best of our knowledge, little is known about counting PED-sets and DIMs in graphs. For the latter, \cite{LinMS14IPL} presents a linear-time algorithm for claw-free graphs. Our PED-set counting algorithm (Section~\ref{sec: weighted and counting version}) can be adapted to count DIMs in cubic time by considering only colorings (d) and (e) in Fig.~\ref{Fig2} and Configuration~I; whenever a black vertex is found, we skip to the next partial $3$-coloring. This problem is NP-complete for many graph classes (cf.~\cite[Table~1]{doFor-Vi-Lin-Lu-Ma-Mo-Sz-2020}).


\end{document}